\documentclass[12pt,epsfig,amsfonts]{amsart}
\usepackage{amsmath,amsthm,amssymb,amscd,epsfig}
\usepackage{graphicx}

\begin{document}
%\date{\today}

%\usepackage{psfrag}\usepackage{varioref}%

%%%%%%%%%%%%%%%%%%%%%%%%%%%%%%
%\newtheorem{theorem}{Theorem}[subsection]
\newtheorem*{theorema}{Theorem}
\newtheorem*{theoremb}{Theorem B}
\newtheorem{definition}{Definition}[section]
\newtheorem{prop}[definition]{Proposition}
\newtheorem{lemma}[definition]{Lemma}
\newtheorem{sublemma}[definition]{Sublemma}
\newtheorem{remark}[definition]{Remark}
\newtheorem{cor}[definition]{Corollary}
\newtheorem{claim}[definition]{Claim}

%%%%%%%%%%%%%%%%%%%%%%%%%%%%%%%

\author{Samuel Senti* and Hiroki Takahasi**}
%\address{Instituto de Matematica, Universidade Federal do Rio
%de Janeiro, C.P. 68 530, CEP 21941-909, R.J., BRASIL}
%\email{senti@im.ufrj.br}
%\address{FIRST, Aihara Innovative Mathematical Modelling Project,
%Japan Science and Technology Agency; Institute of Industrial Science, Universit%y of Tokyo, Tokyo
%153-8505, JAPAN}
% \email{h{\_}takahasi@sat.t.u-tokyo.ac.jp}
\thanks{*Instituto de Matematica, Universidade Federal do Rio
de Janeiro, C.P. 68 530, CEP 21941-909, R.J., BRASIL, senti@im.ufrj.br;
**Department of Electronic Science and Engineering, Kyoto University, Kyoto
606-8501, JAPAN, takahasi.hiroki.7r@kyoto-u.ac.jp
{\it Address from April 1st 2013}: Department of Mathematics, Keio University, Yokohama 223-8522,
JAPAN, hiroki@math.keio.ac.jp
% \email{h{\_}takahasi@sat.t.u-tokyo.ac.jp}
    }

\title[Equilibrium measures
for the H\'enon map at the first bifurcation]
{Equilibrium measures
for the H\'enon map\\ at the first bifurcation}
\begin{abstract}
We study the dynamics of strongly dissipative H\'enon maps, at the first bifurcation parameter where the uniform hyperbolicity is destroyed by the formation of tangencies inside the limit set. We
prove the existence of an equilibrium measure which
minimizes the free energy associated with the non continuous potential $-t\log
J^u$, where $t\in\mathbb R$ is in a certain interval of the form
$(-\infty,t_0)$, $t_0>0$ and $J^u$ denotes the Jacobian in the
unstable direction. \end{abstract}
\maketitle

\section{Introduction}
An important problem in dynamics is to describe how horseshoes are
destroyed. A process of destruction through homoclinic
bifurcations is modeled by the H\'enon family
\footnote{Our arguments and results also hold for 
H\'enon-like families \cite{CLR08,MorVia93}, perturbations of the H\'enon family.}
\begin{equation}\label{henon}
f_{a}\colon(x,y)\mapsto (1-ax^2+\sqrt{b}y, \pm\sqrt{b}x),\ \
0<b\ll1.\end{equation}
For all large $a$, the non-wandering set is a uniformly
hyperbolic horseshoe \cite{DevNit79}. As one decreases $a$, the stable and unstable directions get increasingly confused, and at last reaches a bifurcation parameter $a^*$ near $2$. The non-wandering
set of $f_a$ is a uniformly hyperbolic horseshoe for $a>a^*$, and
$\{f_a\}$ generically unfolds a quadratic tangency at $a=a^*$
\cite{BedSmi04,BedSmi06,CLR08}. According to a general theory of global
bifurcations (for instance, see \cite{PalTak93} and the references
therein), a surprisingly rich array of complicated behaviors
appear in the unfolding of the tangency. In this paper,
instead of unfolding the tangency we study the dynamics of $f_{a^*}$ from a
viewpoint of ergodic theory and thermodynamic formalism.
The
dynamics of $f_{a^*}$ is close to that of the uniformly hyperbolic
horseshoe \cite{BedSmi04,CLR08,Hoe08,Tak12}, yet already exhibits some
complexities shared by those $f_a$, $a<a^*$, and thus will provide
an important insight into the bifurcation at $a^*$.

Another motivation for the study of $f_{a^*}$ is to develop an
ergodic theory for \emph{non-attracting sets which are
not uniformly hyperbolic}. In the rigorous study of dynamical
systems, a great deal of effort has been devoted to the study of
chaotic attractors. A statistical approach has been often taken,
i.e., to look for nice invariant probability measures which
statistically predict the asymptotic ``fate" of positive Lebesgue
measure sets of initial conditions. The non-wandering set of
$f_{a^*}$ behaves like a saddle, in that many orbits wander
around it for a while due to its invariance, and eventually leave a
neighborhood of it \cite{Tak12}. Such
non-attracting sets may be considered somewhat
irrelevant, as they only concern transient behaviors. Although
this point of view is justified by a wide variety of reasons, the
study of non-attracting sets deserves our attention, because
of their nontrivial influences on global dynamics. Moreover, important thermodynamic parameters relevant in this context, such as 
the Hausdorff dimension and
escape rates, are not well-understood unless the uniform hyperbolicity is assumed.

We state our setting and goal in more precise terms.
Write $f$ for $f_{a^*}$ and let $\Omega$ denote the non-wandering set of $f$.
This set is closed, bounded, and so is a compact set. Let $\mathcal M(f)$ denote the space of all
$f$-invariant Borel probability measures endowed with the topology of weak convergence. For a given potential
$\varphi: \Omega\to \mathbb R$ (the minus of) the free energy function
$F_\varphi\colon\mathcal M(f)\to\mathbb R$ is given by
$$F_\varphi(\mu)= h(\mu)+\mu(\varphi),$$ where $h(\mu)$
denotes the entropy of $\mu$ and $\mu(\varphi)=\int\varphi
d\mu$. An \emph{equilibrium measure} for the potential
$\varphi$ is a measure $\mu_\varphi\in\mathcal M(f)$ which maximizes $F_\varphi$, i.e.
$$F_\varphi(\mu_\varphi)=
\sup\left\{F_\varphi(\mu)\colon\mu\in\mathcal M(f)\right\}.$$

The existence and uniqueness of equilibrium measures depend upon the
characteristics of the system and the potential.
In our setting, the entropy map is upper semi-continuous (Corollary \ref{USC})
and so equilibrium measures exist for any continuous potential,
and they are unique for a dense subset of continuous potentials \cite[Corollary 9.15.1]{Wal82}.
However the most significant potentials often lack continuity and the above results do not apply, as is the case of the 
potential we are now going to introduce.

At a point $z\in \mathbb R^2$, let $E^u(z)$ denote the one-dimensional subspace such that \begin{equation}\label{eu}
\varlimsup_{n\to\infty}\frac{1}{n}\log\|Df^{-n}|E^u(z)\|<0.\end{equation}
Since $f^{-1}$ expands area, $E^u(z)$ is unique when it makes sense. We call $E^u$ an {\it unstable direction}.
%We show that $E^u$ is defined everywhere on $K$, is Borel measurable and non continuous at the fixed saddle near $(-1,0)$ (See Proposition \ref{measurable}).
Denote the Jacobian in the unstable direction by 
$$J^u(z):=\Vert Df|E^u(z)\Vert.$$
The \emph{geometric} potential is then given by $$\varphi_t:=-t\log J^u,\ t\in\mathbb R.$$ Due to the presence of the tangency, $\varphi_t$ is merely bounded measurable and not continuous.
Our goal is to prove the existence of equilibrium measures for $\varphi_t$ with $t$ in a certain interval containing all negative $t$ and some (many) positive $t$.

The (non-uniform) expansion along the unstable direction is responsible for the chaotic behavior. Therefore, information on the dynamics of $f$ as well as the geometry of $\Omega$ is obtained by studying equilibrium measures for
the geometric potentials $\varphi_t$ and the associated {\it pressure function} $t\in\mathbb R\mapsto P(t)$, where
$$P(t):=\sup\{F_{\varphi_t}(\mu)\colon
\mu\in\mathcal M(f)\}.$$
For instance, SRB measures when they exist should be equilibrium measures for $\varphi_1$. Those for $\varphi_0$ are the measures of maximal entropy. 
%The {\it pressure function} $t\mapsto P(t)$ is convex, and thus continuous. 
In addition,
analogously to the case of basic sets of $C^2$ surface diffeomorphisms \cite{ManMcC83},  one can show that the Hausdorff dimension of the non-wandering set along the unstable manifold is given by the first zero of the pressure function
\cite[Theorem B]{SenTak12}. As there is no SRB measure for the H\'enon map $f$ at first bifurcation \cite{Tak12}, the dimension is strictly less than $1$.

Our study of $f$ heavily relies on the fact
that $f$ may be viewed as a singular perturbation of the Chebyshev quadratic
$x\in\mathbb R\to 1-2x^2$, because $0<b\ll1$ and $a^*\to2$ as $b\to0$.
Hence, we introduce a small constant $\varepsilon>0$ to quantify %measure
 a proximity of $f$ to the Chebyshev quadratic.
Define $t_0=t_0(\varepsilon,b)$ by
\begin{equation}\label{t0}
t_0=\inf \{t\in\mathbb R\colon P(t)\leq-(t/2)\log(4-\varepsilon)\}.\end{equation}
Observe that 
$0<t_0\leq+\infty$. 
%we shall work with sufficiently small $b>0$ so that any horizontal vector
%at a point outside of a $\delta$-neighborhood of the tangency near the origin gets expanded in norm by a factor in $[2-\varepsilon,4+\varepsilon]$ (See Lemma \ref{outside}). 

\begin{theorema}
For any small $\varepsilon>0$ there exists
$b_0>0$ such that 
if $0<b<b_0$ and $t<t_0$, then
there exists an equilibrium measure for $\varphi_t$.\end{theorema}

The reason for restricting the range of $t$ to values for which the pressure of the system is sufficiently large is to deal with
%from our considerations 
measures which charge the fixed saddle $Q$ (See FIGURE 1) and hence the discontinuity of $\varphi_t$.
The assumption $t<t_0$ 
%in the theorem is used to 
guarantees that such measures are not equilibrium measures for $\varphi_t$.

Let us here mention some previous results closely related to ours which develop thermodynamics of systems 
{\it at the boundary of uniform hyperbolicity}. Makarov $\&$ Smirnov \cite{MakSmi03} studied rational maps on the Riemannian sphere for which every critical point in the Julia set is non-recurrent.
Leplaideur, Oliveira $\&$ Rios \cite{LepOliRio11} 
and 
Arbieto $\&$ Prudente \cite{ArbPru12} studied
 partially hyperbolic horseshoes treated in \cite{DHSR09}.
  Leplaideur $\&$ Rios \cite{LepRio1,LepRio09} proved the existence and 
uniqueness of equilibrium measures for geometric potentials ($t$-conformal measures in their terms), 
for certain type 3 linear horseshoes in the 
plane (horseshoes with three symbols) with a single orbit of tangency studied in \cite{Rio01}. For this model, 
Leplaideur \cite{Lep11} proved the analyticity of the pressure function. 
 Our map $f$ is similar in spirit to the model of \cite{LepRio1,LepRio09} introduced in \cite{Kir96, Rio01}.
However,  different arguments are necessary as $f$ does not satisfy the specific assumptions in \cite{LepRio1,LepRio09},
such as the linearity and the balance between expansion/contraction rates.

%{\bf The global properties of our horseshoe,  such as the symbolic dynamics (the number of symbols), the amount of dissipation and the rate of expansion/contraction of the map, the curvature of the invariant manifolds, are different from 
%those of the horseshoe considered in \cite{LepRio1,LepRio09}.
%With this in mind, it is worthwhile to compare Theorem B with the aforementioned result in 
%\cite{Lep11} which claims that there is only one phase. 
%We conclude that,
%in the destruction of horseshoes through homoclinic bifurcations, 
%the occurrence of phase transitions at the first bifurcation depends upon
%the global properties of the horseshoe.}

\begin{figure}
\begin{center}
\includegraphics[height=6.5cm,width=14cm]{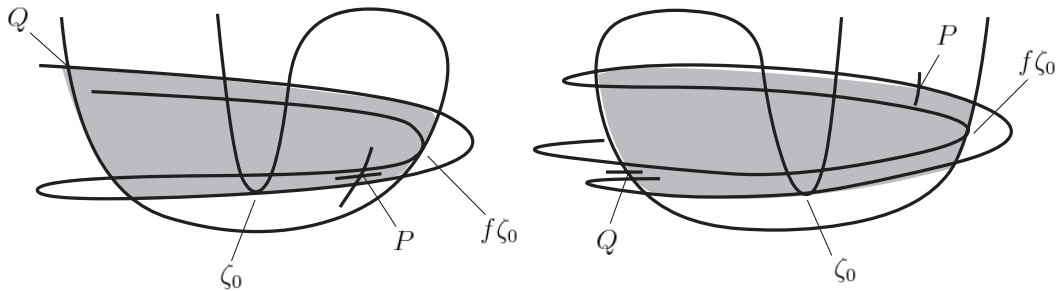}
\caption{Manifold organization for $a=a^*$. There exist two
hyperbolic fixed saddles $P$, $Q$ near $(1/2,0)$, $(-1,0)$
correspondingly. In the orientation preserving case (left),
$W^u(Q)$ meets $W^s(Q)$ tangentially. In the orientation reversing
case (right), $W^u(P)$ meets $W^s(Q)$ tangentially. The shaded
regions represent the region $R$.
 The point of tangency near the origin is denoted by $\zeta_0$ (See Sect.\ref{family}).}
\end{center}
\end{figure}

The main difficulty is to handle 
the limit behaviour of a sequence of Lyapunov exponents. 
For $\mu\in\mathcal M(f)$, let $\lambda^u(\mu)=\mu(\log J^u)$, which we call the {\it unstable Lyapunov exponent} of $\mu$. Since $\log J^u$ is not continuous, the weak convergence $\mu_n\to\mu$ does not imply the convergence $\lambda^u(\mu_n)\to\lambda^u(\mu)$.
We show that entropy and the unstable Lyapunov exponent  are upper semi-continuous as functions of measures 
(Corollary \ref{USC} and Proposition \ref{lyap2}).
Hence, the existence of equilibrium measures for $t\leq0$ follows from the upper semi-continuity of $F_{\varphi_t}$.
For $t>0$ we need a lower bound on the drop $\varliminf{\lambda(\mu_n)}-\lambda(\mu)$, as the unstable Lyapunov exponent may not be lower semi-continuous.

The structure of the paper is as follows. 
In Sect.2 we study the dynamics of $f$. 
%Different arguments from \cite{Lep11,LepRio1,LepRio09} are necessary, because the specific assumptions on their systems 
%(such as the linearity and the balance between expansion and contraction rates)
%do not hold for our $f$.
Our approach follows the well-known line for H\'enon-like systems \cite{BenCar91,MorVia93,WanYou01}, but now for the first bifurcation parameter.
A key ingredient is the notion of critical points (See Sect.\ref{critical}).
In brief terms, these are points where the fold of the map has the most dramatic effect.
To compensate for contractions of derivatives suffered at returns to a critical neighborhood,
we develop a binding argument (Proposition \ref{recovery}). 
In this argument we use a specific feature of the map $f$, namely that all critical points are non recurrent,
which does not hold for the maps treated in \cite{BenCar91,MorVia93,WanYou01}.

In Sect.3 we show that the dynamics on the non-wandering set is semi-conjugate to the full shift on two symbols.
This implies the upper semi-continuity of entropy.
Although this statement is not surprising, standard arguments do not work due to  the presence of the tangency. At the first bifurcation parameter the non-wandering set has a product structure, in the sense that the stable and unstable curves always intersect 
each other at a unique point. This defines the semi-conjugacy. 

In Sect.4 we use the results in Sect.2 to bound the amount of drop of the unstable Lyapunov exponents of sequences of measures (Proposition \ref{lyap2}). Using this bound and the assumption $t<t_0$, i.e., the pressure $P(t)$ is sufficiently large,
we  complete the proof of the theorem. 
In Appendix we show that $t_0$ can be made arbitrarily large by choosing small $\varepsilon$ and $b$.

\section{The dynamics}
In this section we study the dynamics of $f$. In Sect.\ref{family} we state and prove basic geometric properties
surrounding the invariant manifolds of fixed saddles.
Although the dynamics outside of a fixed neighborhood of the point of tangency is uniformly hyperbolic, returns to this neighborhood is unavoidable. To control
these returns, in Sect.\ref{critical} we introduce {\it critical points} following the idea of Benedicks $\&$ Carleson \cite{BenCar91}.
In Sect.\ref{recover} we analyze the dynamics near the orbits of the critical points. In Sect.\ref{uleaf} and Sect.\ref{bf} we discuss how to associate critical points to generic orbits which fall inside the neighborhood of the tangency.

We use several positive constants whose purposes are as follows:

\begin{itemize}
\item $\varepsilon,\delta,b$ are small constants, chosen in this order;
$\varepsilon$ is the constant specified in the theorem;
$\delta$ is used to define a critical region (See Sect.\ref{critical});
$b$ is the constant from \eqref{henon}.
We may shrink $\delta$ and $b$ if necessary, but only a finite number of times;

\item three constants below are used for estimates of derivatives:
\begin{equation}\label{sigma12}
\sigma=2-\frac{\varepsilon}{2},\  
\lambda_1=4-\frac{\varepsilon}{2},\ \lambda_2=4+\frac{\varepsilon}{2};
\end{equation}
The $\sigma$ is used as a lower bound for derivatives far away from a critical region;
$\lambda_1$, $\lambda_2$ are used as a lower and upper bounds for derivatives
near the fixed saddle near $(-1,0)$.

\item any generic constant independent of $\varepsilon$, $\delta$, 
$b$ is simply denoted by $C$.
\end{itemize}

\subsection{Basic geometric properties of the invariant manifolds}\label{family}
Let $P$, $Q$ denote the
fixed saddles near $(1/2,0)$ and $(-1,0)$ respectively.
 If $f$
preserves orientation, let $W^u=W^u(Q)$. If $f$ reverses
orientation, let $W^u=W^u(P)$. %For $z\in W^u$,
%let $t(z)$ denote any unit vector
%tangent to $W^u$ at $z$.
By a {\it rectangle} we mean any
closed region bordered by two compact curves in $W^u$ and two in the
stable manifolds of $P$, $Q$. By an {\it unstable side} of a
rectangle we mean any of the two boundary curves in $W^u$. A {\it
stable side} is defined similarly.

Let $R$ denote the largest possible rectangle determined by $W^u$ and $W^s(P)$, as indicated in Figure 1.
One of its unstable sides of $R$ contains the point of tangency near $(0,0)$, which we denote by $\zeta_0$.
Let $\alpha_0^+$ denote the stable side of $R$ containing $f\zeta_0$ and let $\alpha_0^-$ denote the other stable side of $R$. 
Since any point outside of $R$ diverges to infinity under positive or negative iteration \cite{CLR08},
the non-wandering set $\Omega$ is contained in $R$.

Let $S$ denote the closed lenticular region bounded by the unstable side of $R$ and the parabola in $W^s(Q)$ containing $\zeta_0$. Points in the interior of $S$ is mapped to the outside of $R$, and they never return to $R$ under any positive iteration.

\begin{figure}
\begin{center}
\includegraphics[height=6.5cm,width=14cm]{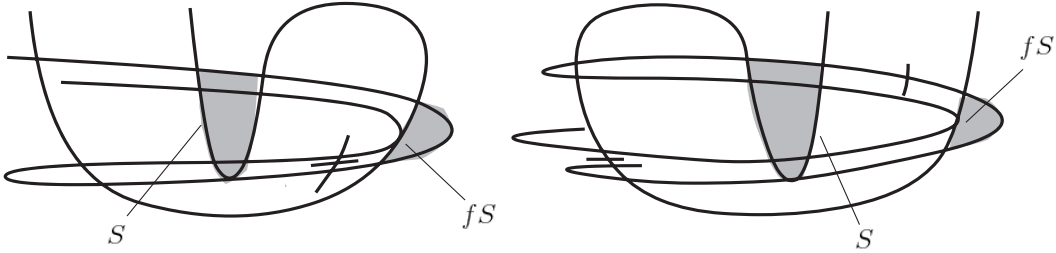}
\caption{The shaded closed lenticular region is denoted by $S$ (left: orientation preserving case; right: orientation reversing case).
 The interior of $S$ is mapped to the outside of $R$, and its forward iterates do not intersect $R$.}
\end{center}
\end{figure}

We need a couple of lemmas on the geometry of $W^u$.
Let $\alpha_1^+$ denote the component of $W^s(P)\cap R$ containing $P$. 
Let $\alpha_1^-$ denote the one of the two components of $R\cap f^{-1}\alpha_1^+$ which lies at the left of $\zeta_0$.
Let $\Theta$ denote the rectangle bordered by $\alpha_1^+$, $\alpha_1^-$ and the unstable sides of $R$.
The next lemma roughly states that `folds" in $W^u$ do not enter $\Theta$. 
By a {\it $C^2(b)$-curve} we mean
a closed curve for which the
slopes of its tangent directions are $\leq b^{\frac{1}{4}}$ and the
curvature is everywhere $\leq b^{\frac{1}{4}}$.
\begin{lemma}\label{c1}\cite[Section 4]{Tak12}
Any component of $\Theta\cap W^u$ is a $C^2(b)$-curve with endpoints in $\alpha_1^-,\alpha_1^+$.
\end{lemma}
The next lemma will not be used for some time.
For $k\geq0$, let $\Delta_k=\Theta\cap f^kR$.
Observe that $\Delta_k$ has $2^k$ components each of which is a
rectangle, and by Lemma \ref{c1}, the unstable sides of it are
$C^2(b)$-curves. Also observe that $\Delta_{k}$ is related to
$\Delta_{k-1}$ as follows: let $\mathcal Q_{k-1}$ denote any
component of $\Delta_{k-1}$. Then $\mathcal Q_{k-1}\cap f^kR$ has
two components, each of which is a component of $\Delta_{k}$.

\begin{lemma}\label{geo} For $k=0,1,\ldots$ and for each component
$\mathcal Q_{k}$ of $\Delta_k$, the Hausdorff
distance between its unstable sides is $\mathcal
O(b^{\frac{k}{2}})$.
\end{lemma}
\begin{proof}
We argue by induction on $k$. Assume the statement for
$0\leq k<j$. We regard the unstable sides of $\mathcal
Q_{j}$ as graphs of functions $\gamma_1$, $\gamma_2$ defined on an
interval $I$. Let $L(x)=|\gamma_1(x)-\gamma_2(x)|.$ Since
$\mathcal Q_j$ is contained in a component of $\Delta_{j-1}$, the
assumption of induction gives $L^{\frac{1}{2}}(x)\leq
(Cb)^{\frac{j-1}{4}}<{\rm length}(I).$ Moreover
$|\gamma_1'(x)-\gamma_2'(x)|\leq L^{\frac{1}{2}}(x)$ holds, since $\gamma$ is $C^2$ and so
otherwise $\gamma_1$ would intersect $\gamma_2$. By this and the
definition of $C^2(b)$-curves, $L(y)\geq L(x)
-(L^{\frac{1}{2}}(x)-Cb^{\frac{1}{4}}|x-y|)|x-y|$ holds for $x,y\in I$,
which is $\geq L(x)/2$ provided  $|x-y|\leq L^{\frac{2}{3}}(x)$.
Hence, ${\rm area}(\mathcal Q_{j})\geq L^{\frac{5}{3}}(x)/2$
holds. If $L(x)\geq b^{\frac{j}{2}},$ then ${\rm area}(\mathcal
Q_{j})\geq b^{\frac{5j}{6}}/2$, which yields a contradiction to
${\rm area}(\mathcal Q_{j})<{\rm area}(f^jR)\leq (Cb)^j$.
\end{proof}

\subsection{Critical points}\label{critical}
%{\bf In this subsection
%we first note that the dynamics is hyperbolic in most of the phase space, outside of a small neighborhood
%of the tangency $\zeta_0$. We then introduce critical points to control returns to this small neighborhood.}

We introduce a small neighborhood of the tangency $\zeta_0$ as follows.
Define
$$I(\delta)=(-\delta,\delta)\times(-b^{\frac{1}{4}}, b^{\frac{1}{4}}).$$
Observe that, for any given $\delta>0$, $\zeta_0\in I(\delta)$ provided $b$ is sufficiently small.

The next lemma, which controls the growth of horizontal 
vectors outside of a fixed neighbourhood of the tangency, readily follows from viewing $f$ 
as a perturbation of the Chebyshev quadratic which is smoothly conjugate to the tent map. We say a nonzero tangent vector $v$ is {\it $b$-horizontal} if ${\rm slope}(v)\leq b^{\frac{1}{4}}$.

\begin{lemma}\label{outside}
For any $\varepsilon>0$, $\delta>0$ there exists $b_0=b_0(\varepsilon,\delta)>0$ such that
the following holds for all $0<b<b_0$:

\begin{itemize}

\item[(a)] if $n\geq1$ and $z\in R$ is such that $z,fz,\ldots,f^{n-1}z \notin I(\delta)$,
then for any $b$-horizontal vector $v$ at $z$, $Df^n(z)v$ is $b$-horizontal and
$\|Df^n(z)v\|\geq \delta\sigma^n\|v\|$. If moreover $f^nz\in I(\delta)$, then $\|Df^n(z)v\|\geq \sigma^n\|v\|$;

 \item[(b)] if $z\in [-2,2]^2\setminus\Theta$, then for any $b$-horizontal vector $v$ at $z$, 
 $Df(z)v$ is $b$-horizontal and $\|Df(z)v\|\geq \sigma\|v\|$. \end{itemize}
\end{lemma}

By virtue of Lemma \ref{outside}, the dynamics outside of the fixed neighborhood $I(\delta)$ is uniformly hyperbolic.
To recover the loss of hyperbolicity due to returns to the inside of $I(\delta)$, we mimic the strategy of Benedicks \& Carleson \cite{BenCar91}
and develop a binding argument relative to {\it critical points}.
For the rest of this subsection we introduce critical points, and perform preliminary estimates
needed for the binding argument in the next subsection.

\begin{figure}
\begin{center}
\includegraphics[height=4cm,width=8.3cm]{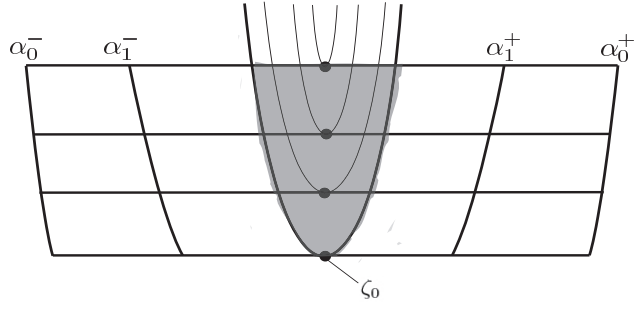}
\caption{Thick segments are part of $W^u$ and $W^s(P)$, $W^s(Q)$. The shaded region is $S$. 
The dots represent the critical points on $\Theta\cap W^u$. The parabolas represent the pull-backs of the leaves of $\mathcal F^s$.}
\end{center}
\end{figure}

From
the hyperbolicity of the saddle $Q$ it follows that (use 
the Center Manifold Theorem \cite{Shu87} for the tangent bundle map)
there exist two mutually disjoint connected open sets
$U^-$, $U^+$ independent of $b$ such that $\alpha_0^-\subset U^-$,
$\alpha_0^+\subset U^+$, $U^+\cap  fU^+=\emptyset=U^+\cap  fU^-$
and
a foliation $\mathcal F^s$ of $U:=U^-\cup U^+$ by one-dimensional leaves such
that:
\begin{itemize}
\item[(F1)] $\mathcal F^s(Q)$, the leaf of $\mathcal F^s$ containing $Q$,
contains $\alpha_0^-$; 
\item[(F2)] if $z,fz\in U$, then $f(\mathcal F^s(z))
\subset\mathcal F^s(fz)$;

\item[(F3)] Let $e^s(z)$ denote the unit vector in $T_z\mathcal F^s(z)$ with the
positive second component. Then: $z\to e^s(z)$ is $C^{1}$ and $\|Dfe^s(z)\|\leq Cb$, $\|\frac{\partial}{\partial z}e^s(z)\|\leq C$;

\item[(F4)] If $z,fz\in U$, then ${\rm slope}(e^s(z))\geq
C/\sqrt{b}.$
\end{itemize}
We call $\mathcal F^s$ a {\it stable foliation} on $U$.
From (F1), (F2) and $f\alpha_0^+\subset\alpha_0^-$ 
it follows that there is a leaf of $\mathcal F^s$ which 
contains $\alpha_0^+$. 
(F4) can be checked by
contradiction: if it were false, then ${\rm slope}(e^s(fz))\ll1.$

%\item[(F2)] each leaf intersects $W^u_{\rm loc}(Q)$ exactly at one
%point, and the length of each leaf is $\geq C$;

%\begin{definition}\label{definitioncritical}
We say $\zeta\in W^u$ is a {\it critical point} if $f\zeta\in U^+$ and 
$T_{f\zeta}W^u=T_{f\zeta}\mathcal F^s(f\zeta)$.
%\end{definition}
From the results in \cite{Tak12} it follows that any component of $\Theta\cap W^u$ admits a unique
critical point, and it is contained in $S$. Hence:

\begin{itemize}

\item $\Omega$ does not contain any critical point other than $\zeta_0$;

\item any critical point other than $\zeta_0$ is mapped by $f$ to the outside of $R$, and then escapes
to infinity under positive iteration. 
\end{itemize}

The second property implies that the critical orbits are contained in a region 
where the uniform hyperbolicity is apparent. Hence, by binding generic orbits which fall inside $I(\delta)$
to suitable critical orbits, and then copying the exponential growth along the critical orbits, one shows that the horizontal slopes and the expansion are restored after suffering from the loss due to  the folding behavior near $I(\delta)$.
The time necessary for this recovery is called {\it bound periods}, introduced in Sect.\ref{recover}. This type of {\it binding argument} traces back to Jakobson \cite{Jak81} and Benedicks $\&$ Carleson \cite{BenCar85,BenCar91}.
Our binding argument is an extension of Benedicks $\&$ Carleson's to the first bifurcation parameter $a^*$
which is not treated in \cite{BenCar91}.
 %In this extension, new considerations are necessary because of the escaping property of the critical points, which is not the case in \cite{BenCar91}.}

The escaping property motivates the following definition.
For a critical point $\zeta$ define
$$n(\zeta)=\sup\{i\geq1\colon f^i\zeta\in U\}.$$ %is called an {\it escaping time of $\zeta$.}
We have $n(\zeta)\in[1,+\infty]$, and $n(\zeta)=+\infty$ if and only if $\zeta=\zeta_0$. For $i\geq1$ let $w_i(\zeta)=Df^{i-1}(f\zeta)\left(\begin{smallmatrix}
1\\0\end{smallmatrix}\right)$. 
Since 
all forward iterates of $\zeta$ up to time $n(\zeta)$ are near the stable sides of $R$,
for every $1\leq i\leq n(\zeta)$ we have
$${\rm slope}(w_i(\zeta))\leq\sqrt{b},$$
and 
\begin{equation}\label{misiurewicz}
\lambda_1\|w_i(\zeta)\|\leq \|w_{i+1}(\zeta)\|\leq
5\|w_i(\zeta)\|.\end{equation}

For $r>0$ let $$B(r)=\{x\in\mathbb R^2\colon \min\{|x-y|\colon y\in\alpha_0^-\cup\alpha_0^+\}\leq r\}.$$
Choose a constant $\tau>0$ independent of $b$ such that 
\begin{equation}\label{TAU} 
B(22\tau)\subset U\ \ \text{and}\ \ \tau\leq\frac{\sigma}{100}\log2.\end{equation}
For $p\in[1,n(\zeta)]$
define
\begin{equation}\label{Theta}D_p(\zeta)=\tau\left[\sum_{i=1}^{p}
d_i^{-1}(\zeta)\right]^{-1},\ \
\text{where}\ \
d_i(\zeta)=\frac{\|w_{i+1}(\zeta)\|}{\|w_{i}(\zeta)\|^2}.\end{equation}
The number $D_p(\zeta)$ serves to define a strip around the leaf $\mathcal F^s(f\zeta)$ on which 
the distortion of $f^{p-1}$ is controlled (see Lemma \ref{dist} and \eqref{dist1}).
The next lemma gives estimates on the size of this strip and its $f^{p-1}$-iterates.

\begin{lemma}\label{size2}
There exists $p_0=p_0(\varepsilon)$ such that if $p\geq p_0$, then
\begin{itemize}

\item[(a)] $(\lambda_2+\varepsilon/2)^{-p}\leq D_{p}(\zeta)\leq \lambda_1^{-p};$

\item[(b)] $\tau/5\leq\|w_{p}(\zeta)\|D_{p}(\zeta)\leq 5\tau.$
\end{itemize}
\end{lemma}
\begin{proof}
\eqref{misiurewicz} yields
$$\left(\lambda_2+\frac{\varepsilon}{2}\right)^{-p}\leq \frac{\tau}{p}\cdot\lambda_1\lambda_2^{-p+1}\leq\frac{\tau}{p}\cdot\min_{1\leq i\leq p}d_i(\zeta)\leq D_p(\zeta)\leq \tau d_p(\zeta)\leq 5\tau\lambda_1^{-p+1}\leq\lambda_1^{-p}.$$
The first inequality holds for sufficiently large $p$ depending only on $\varepsilon$.
As for (b) we have
\[\|w_{p}(\zeta)\|D_{p}(\zeta)<\tau\|w_p(\zeta)\|d_p(\zeta)
=\tau
\frac{\|w_{p+1}(\zeta)\|}{\|w_{p}(\zeta)\|}\leq 5\tau.\]
For the lower estimate, \eqref{misiurewicz} yields
\[\frac{1}{\|w_{p}(\zeta)\|D_{p}(\zeta)}=
\frac{1}{\tau}\sum_{i=1}^p\frac{\|w_{i}(\zeta)\|}{\|w_{p}(\zeta)\|}
\frac{\|w_{i}(\zeta)\|}{\|w_{i+1}(\zeta)\|}\leq\frac{1}{\tau}\sum_{i=1}^p\lambda_1^{-(p-i+1)}\leq
\frac{5}{\tau}.\qedhere\]
\end{proof}

%In what follows we assume $\delta$ and $b$ are sufficiently small,
%so that any critical point is non escaping up to time $n_0$.

\subsection{Recovering Hyperbolicity}\label{recover}
 We now develop a binding argument for the map $f$ at the first bifurcation in order to recover hyperbolicity.
Throughout this subsection we assume $\zeta$ is a critical point, and 
$\gamma$ is a $C^2(b)$-curve in $I(\delta)$ which contains $\zeta$ and is tangent to $E^u(\zeta)$.
Consider the leaf $\mathcal F^s(f\zeta)$ of the stable foliation $\mathcal F^s$ through $f\zeta$.
This leaf
may be expressed as a graph of a smooth function:
there exists an open interval $J$ independent of $b$ and a smooth function $y\mapsto x(y)$ on $J$ such that
 $$\mathcal F^s(f\zeta)=\{(x(y),y)\colon y\in J\}.$$
For a point $z\in\gamma\setminus\{\zeta\}$ we associate two integers $p(z)\in[1,n(\zeta)]$, $q(z)\in[1,n(\zeta)]$ called {\it bound} and {\it fold periods} as follows.
First, let $p=p(z)$ be such that
\begin{equation}\label{bddperiod}fz\in\left\{(x,y)\colon
 D_{p}(\zeta)< |x-x(y)|\leq D_{p-1}(\zeta),\ y\in J\right\},\end{equation}
when it makes sense. Next,  define $q=q(z)$ by
\begin{equation}\label{q}q=\min\left\{1\leq i<p\colon|\zeta-z|
^{\beta}  \|w_{j+1}(\zeta)\|\geq1\ \ \text{for every}\ \ i\leq
j<p\right\},\end{equation}
where \begin{equation}\label{beta}\beta=2/\log\left(1/b\right).\end{equation}
Note that \eqref{bddperiod} \eqref{beta} yield $|\zeta-z|
^{\beta}  \|w_{p}(\zeta)\|\geq1$. So, $q$ makes sense
when $p$ does, and $q\leq p-1$.
Also, note that if $\zeta=\zeta_0$, then $p$ makes sense for all $z\in\gamma\setminus\{\zeta\}$
because $n(\zeta_0)=+\infty$. Otherwise, $p$ does not make sense when $z$ is
too close to $\zeta$.

The purposes of these two periods are as follows:
the fold period is used to restore large slopes of iterated tangent vectors to small slopes; the bound period is used to recover an expansion of derivatives.

We are in position to state a result we are leading up to. Let us agree that, for two
positive numbers $A$, $B$, $A\approx B$ indicates $1/C\leq a/b\leq C$
for some $C\geq1$ independent of $\varepsilon$, $\delta$, $b$.
\begin{prop}\label{recovery}
Let $\zeta$ be a critical point, and $\gamma$ a $C^2(b)$-curve in $I(\delta)$
which contains $\zeta$ and is tangent to $E^u(\zeta)$.
%\marginpar{Sam: note that I changed $E^u$ to $W^u$ in order to use the same notation as on line 2 of this section\ref{recover}
%Hiroki: it has to be $E^u(\zeta)$. I have changed the notation on line 2 to $E^u(\zeta)$.}
If $z\in\gamma\setminus\{\zeta\}$ and $p$, $q$ are the corresponding bound and fold periods, then:
\begin{itemize}
\item[(a)]
 $\log|\zeta-z|^{-\frac{2}{\log5}}\leq p\leq\log
|\zeta-z|^{-\frac{3}{\log\lambda_1}};$
\item[(b)]  $\log|\zeta-z|^{-\frac{\beta}{\log\lambda_2}}\leq
 q\leq\log|\zeta-z|^{-\frac{\beta}{\log\lambda_1}}+1$.
\end{itemize}
Let $v(z)$ denote any unit vector tangent to $\gamma$ at $z$. Then:
\begin{itemize}
\item[(c)]
 $\|Df^iv(z)\|\approx
|\zeta-z|\cdot\Vert w_i(\zeta)\Vert$ for every $q< i\leq p$;
\item[(d)] $\|Df^iv(z)\|<1$ for every $1\leq i<q$;
\item[(e)] $\displaystyle{\|Df^pv(z)\|\geq (4-\varepsilon)^{\frac{p}{2}}};$
\item[(f)] ${\rm slope}(Df^pv(z))\leq b^{\frac{1}{4}}$.
\end{itemize}
\end{prop}

A proof of this proposition follows the line
\cite{BenCar91,MorVia93,WanYou01} that is now well-understood. We split
$Dfv(z)$ into $\left(\begin{smallmatrix}
1\\ 0\end{smallmatrix}\right)$-component and $e^s(fz)$-component, and
iterate them separately.
The latter is contracted exponentially, and the former copies the
growth of $w_1(\zeta),\ldots,w_{p}(\zeta)$, and so is expanded exponentially.
The contracted component is eventually dominated by the expanded one,
and as a result the desired estimates holds.

%\begin{lemma}\label{hyp}
%If $n\geq1$ and $\gamma$ is a $C^2(b)$-curve such that $\gamma,
%f\gamma,\ldots,f^{n-1}\gamma\subset [-2,2]^2\setminus\Theta$, then:
%\begin{itemize}

%\item[(a)] $f^n\gamma$ is a $C^2(b)$-curve;

%\item[(b)]
%for all $\xi,\eta\in\gamma$, $\displaystyle{\log\frac{\|Df^nv(\xi)\|}{\|Df^nv(\%eta)\|}\leq
%8|f^n\xi-f^n\eta|}$, where $v(z)$ denotes any unit vector tangent to $\gamma$ a%t $z$.
%\end{itemize}
%\end{lemma}
%\begin{proof}
%(a) is
%immediate from the form of our map (\ref{henon}) and the definition of 
%$\Theta$. As for (b), using Lemma \ref{outside}(b), $\|Df\|\leq 9$, $\|D^2f\|\l%eq 5$ on $[-2,2]^2$
%and the fact that $f^i\gamma$ is $C^2(b)$ for every $0\leq i<n$, 
% \begin{align*}
%\|Dfv(f^i\xi)-Dfv(f^i\eta)\|&\leq \|Df(f^i\xi)\|\|v(f^i\xi)-v(f^i\eta)\|+\|Df(f%^i\xi)-Df(f^i\eta)\|\\&\leq 14|f^i\xi-f^i\eta|\leq 14\sigma^{i-n}|f^n\xi-f^n%\eta|.\end{align*}
%Hence we obtain
%\begin{align*}
%\log\frac{\|Df^nv(\xi)\|}{\|Df^nv(\eta)\|}&=\sum_{i=0}^{n-1}\log\frac{\|Df^iv(\%xi)\|}{\|Df^iv(\eta)\|}\leq \frac{1}{\sigma}\sum_{i=0}^{n-1}\|Dfv(f^i\xi)-Df%v(f^i\eta)\|\\
%&\leq\frac{14}{\sigma}\sum_{i=0}^{n-1}|f^i\xi-f^i\eta|\leq
% \frac{14}{\sigma}|f^n\xi-f^n\eta|\sum_{i=0}^{n-1}\sigma^{i-n}\\
%&\leq \frac{14}{\sigma^2-\sigma}|f^n\xi-f^n\eta|<8|f^n\xi-f^n\eta|.
%\end{align*}
%\end{proof}

The proof of Proposition \ref{recovery} will be given after the next

\begin{lemma}\label{dist}
Let $(x(y_0),y_0)\in\mathcal F^s(f\zeta)$, and let $\gamma_0$ be the horizontal segment of the form $\gamma_0=\{(x,y_0)\colon |x-x(y_0)|\leq D_{p-1}(\zeta)\}.$
Then:
\begin{itemize}
\item[(a)] for all $\xi,\eta\in \gamma_0$ and every $1\leq i<p$, $\|Df^{i}(\xi)\left(\begin{smallmatrix}1\\0\end{smallmatrix}\right)\|\leq2\cdot\|Df^{i}(\eta)\left(\begin{smallmatrix}1\\0\end{smallmatrix}\right)\|$;
\item[(b)] for every $1\leq i<p$, $f^i\gamma_0$ is a $C^2(b)$-curve and ${\rm length}(f^{i}\gamma_0)\leq 20\tau$.
\end{itemize}
\end{lemma}
\begin{proof}
 These estimates would hold
if for all $0\leq j <p-1$ we have
\begin{equation}\label{disteq}
f^j\gamma_0\subset[-2,2]^2\setminus\Theta,\ \ \  {\rm length}(f^j\gamma_0)\leq 
20d_{j+1}^{-1}(\zeta)D_{p-1}(\zeta)\leq 20\tau.
\end{equation}
Indeed, let $1\leq i<p$. Summing the inequality in \eqref{disteq} over all $j=0,1,\ldots,i-1$ yields

\begin{align*}
\log\frac{\|Df^i(\xi)\left(\begin{smallmatrix}1\\0\end{smallmatrix}\right)\|}{\|Df^i(\eta)\left(\begin{smallmatrix}1\\0\end{smallmatrix}\right)\|}&=\sum_{j=0}^{i-1}\log\frac{\|Df^j(\xi)\left(\begin{smallmatrix}1\\0\end{smallmatrix}\right)\|}{\|Df^j(\eta)\left(\begin{smallmatrix}1\\0\end{smallmatrix}\right)\|}\leq \frac{1}{\sigma}\sum_{j=0}^{i-1}\|Df(f^j\xi)\left(\begin{smallmatrix}1\\0\end{smallmatrix}\right)-Df(f^j\eta)\left(\begin{smallmatrix}1\\0\end{smallmatrix}\right)\|\\
&\leq\frac{5}{\sigma}\sum_{j=0}^{i-1}{\rm length}(f^j\gamma_0)\leq
 \frac{100\tau}{\sigma}\leq\log2,
\end{align*}
where the last inequality follows from the second condition in \eqref{TAU}.

We prove (\ref{disteq}) by induction on $j$. It is immediate to
check it for $j=0$. Let $k>0$ and assume
(\ref{disteq}) for every $0\leq j < k$.
Then, from the form of our map \eqref{henon}, 
$f^{k}\gamma_0$ is a $C^2(b)$-curve. %and 
%${\rm length}(f^k\gamma_0)\leq 100\tau$.
Summing the inequality in \eqref{disteq} over
all $0\leq j<k$ and then using \eqref{TAU} yields
 $\Vert
Df^k(\xi)\left(\begin{smallmatrix}
1\\
0\end{smallmatrix}\right)\Vert\leq 2
\cdot\Vert Df^k(\eta)\left(\begin{smallmatrix}
1\\
0\end{smallmatrix}\right)\Vert$ 
for all $\xi,\eta\in \gamma_0$. 
By a result of \cite[Section 6]{MorVia93}, $\Vert
Df^k(z_0)\left(\begin{smallmatrix}
1\\
0\end{smallmatrix}\right)\Vert\leq 2
\cdot\Vert Df^k(f\zeta)\left(\begin{smallmatrix}
1\\
0\end{smallmatrix}\right)\Vert$,
where $z_0=(x(y_0),y_0)$.
 Hence
\begin{align*}
{\rm length}(f^{k}\gamma_0)&\leq 4 \|w_{k+1}(\zeta)\| D_{p-1}(\zeta) = 4
d_{k+1}^{-1}(\zeta)D_{p-1}(\zeta)\frac{\|w_{k+2}(\zeta)\|}{\|w_{k+1}(\zeta)\|}\\
&\leq20
d_{k+1}^{-1}(\zeta)D_{p-1}(\zeta)\leq  20\tau.
\end{align*}
Since $k<p\leq n(\zeta)$, $f^k\zeta\in B(\tau)$ and thus $f^{k}\gamma_0\subset[-2,2]^2\setminus\Theta$ holds. 
Hence \eqref{disteq} holds for $j=k$.
\end{proof}

\medskip

\noindent{\it Proof of Proposition \ref{recovery}.}
Split $$Dfv(z)=A_0\cdot
\left(\begin{smallmatrix}1\\0\end{smallmatrix}\right)+
B_0\cdot e^s(f\zeta).$$
 By \cite[Lemma 2.2]{Tak11},
\begin{equation}\label{quadratic}|A_0|\approx|\zeta-z|\ \ \text{and} \ \ |B_0|\leq Cb^{\frac{1}{4}}.\end{equation}
For a point $r$ near $f\zeta$, write
$r=f\zeta+\xi(r)w_1(\zeta)^\top+\eta(r)e^s(f\zeta)^\top,$ where $\top$
denotes the transpose. The integrations of the inequalities in
(\ref{quadratic}) along $\gamma$ from $\zeta$ to $z$ give
\begin{equation}\label{dist4}
|\xi(fz)|\approx|\zeta-z|^2\ \ \text{and} \ \ |\eta(fz)|\leq Cb^{\frac{1}{4}}|\zeta-z|.\end{equation} Write $fz=(x_0,y_0)$ and $f\zeta=(x_1,y_1)$. Since $f\gamma$
is tangent to $\mathcal F^s(f\zeta)$ at $f\zeta$ we have
$\frac{d\xi(x(y),y)}{dy}(y_1)=0$. (F1) gives
$\left|\frac{d^2\xi(x(y),y)}{dy^2}\right|\leq C.$ %\sqrt{b}.$$
Then
$$|\xi(x(y_0),y_0)|\leq C|y_0-y_1|^2.$$
(\ref{dist4}) gives
$$|y_0-y_1|^2\leq
C|\eta(fz)|^2 \leq C\sqrt{b}|\zeta-z|^2.$$  Since
$|x_0-x(y_0)|=|\xi(fz)-\xi(x(y_0),y_0)|$, the above two inequalities and (\ref{dist4}) yield
\begin{equation}\label{square0}
|x_0-x(y_0)|\approx|\zeta-z|^2.
\end{equation}
Using \eqref{bddperiod} \eqref{square0} and Lemma~\ref{size2}(a) we have
\begin{equation}\label{squaredistance}
|\zeta-z|^2\leq
C\cdot D_{p-1}(\zeta)\leq C\cdot\lambda_1^{-p}
\end{equation}
Taking logs, rearranging the results and then shrinking $\delta$ if necessary we get
$$p\log\lambda_1\leq \log C-2\log|\zeta-z|\leq-3\log|\zeta-z|,$$
which yields the upper estimate in (a).
 For the lower one, using
\eqref{bddperiod} \eqref{square0} and Lemma~\ref{size2}(a) again we have
 \begin{equation}\label{squaredistance2}|\zeta-z|^2\geq
C\cdot D_{p}(\zeta)\geq C\cdot (\lambda_2+\varepsilon/2)^{-p}.\end{equation}
Taking logs of both sides, rearranging the results and then shrinking $\delta$ if necessary we get
$$-2\log|\zeta-z|\leq p\log(\lambda_2+\varepsilon/2)+\log C\leq p\log 5.$$
The last inequality is due to the fact that
the lower bound of $p$ becomes larger as $\delta$ gets smaller. This completes the proof of (a).

As for (b),
\eqref{misiurewicz} and the definition of $q$ give
$$\lambda_1^{q-1}\leq \|w_{q}(\zeta)\|<|\zeta-z|^{-\beta}.$$
Taking logs of both sides and then rearranging the result
yields the upper estimate in (b). For the lower one,
using \eqref{misiurewicz} and the definition of $q$ again we have
$$\lambda_2^q\geq\|w_{q+1}(\zeta)\|\geq|\zeta-z|^{-\beta}.$$
Taking logs of both sides yields the lower estimate in (b).
This completes the proof of (b).

Before proceeding further, we establish a bounded distortion in the strip 
\begin{equation}\label{strip}\left\{(x,y)\colon
|x-x(y)|\leq D_{p-1}(\zeta),y\in J\right\}.\end{equation}
Take arbitrary two points
$\xi_1$, $\xi_2$ in the strip \eqref{strip}, and denote by
$\eta_{\sigma}$ the point of $\mathcal
F^s(f\zeta)$ with the same $y$-coordinate as that of $\xi_\sigma$
($\sigma=1,2$).
By the result of \cite[Section 6]{MorVia93}, $\Vert
Df^i(\eta_1)\left(\begin{smallmatrix}
1\\
0\end{smallmatrix}\right)\Vert\leq 2\cdot\Vert Df^i(\eta_2)\left(\begin{smallmatrix}
1\\
0\end{smallmatrix}\right)\Vert$ holds
for every $1\leq i< p$.
This and Lemma \ref{dist}(a) yield
\begin{equation}\label{dist1}\Vert Df^i(\xi_1)\left(\begin{smallmatrix}
1\\
0\end{smallmatrix}\right)\Vert\leq 8\cdot\Vert
Df^i(\xi_2)\left(\begin{smallmatrix}
1\\
0\end{smallmatrix}\right)\Vert\ \ 1\leq\forall i< p.\end{equation}

We now move on to proving the rest of the items of Proposition \ref{recovery}. 
Consider another splitting
$$Dfv(z)=A\cdot
\left(\begin{smallmatrix}1\\0\end{smallmatrix}\right)+B\cdot e^s(fz),$$
and write $$e^s(fz)=\begin{pmatrix}\cos\theta(z)\\\sin\theta(z)
\end{pmatrix}\ \
\text{and} \ \ \rho\cdot Dfv(z)=\begin{pmatrix}\cos\psi\\\sin\psi
\end{pmatrix},$$ where
$\theta,\psi\in[0,\pi)$ and $\rho>0$ is the normalizing constant.
(\ref{quadratic}) implies $|\theta(\zeta)-\psi| \approx
\rho^{-1}|\zeta-z| \gg|\zeta-z|.$ (F1) gives
$|\theta(\zeta)-\theta(z)|\leq
C|\zeta-z|\ll|\theta(\zeta)-\psi|,$ which implies
$|\theta(z)-\psi|\approx |\theta(\zeta)-\psi|.$ Hence
\begin{equation}\label{recoveq}
|A|\approx
\rho|\theta(z)-\psi|\approx\rho|\theta(\zeta)-\psi|\approx
|\zeta-z|.\end{equation}
Using \eqref{dist1} \eqref{recoveq} we have
\begin{equation}\label{Aexpand}
|A|\cdot\Vert Df^{i-1}(fz)\left(\begin{smallmatrix}
1\\
0\end{smallmatrix}\right)\Vert
\approx |\zeta-z| \cdot \|w_{i}(\zeta)\|.\end{equation}

If $i>q$, then we have
 \begin{equation}\label{Aexpan}
|A|\cdot\Vert Df^{i-1}(fz)\left(\begin{smallmatrix}
1\\
0\end{smallmatrix}\right)\Vert
\approx |\zeta-z| \cdot \|w_{i}(\zeta)\|>
|\zeta-z|^{1-\beta},\end{equation} 
and
\begin{equation}\label{Bcontract}
|B|\cdot\Vert Df^{i-1}e^s(fz)\Vert\leq
(Cb)^{i-1}\leq (Cb)^q \leq
|\zeta-z|^{\frac{3}{2}}.
\end{equation}
The  inequality in \eqref{Aexpan} follows from the definition of $q$. 
The last inequality in \eqref{Bcontract} follows from the
lower estimate of $q$, the definition of $\beta$ and Proposition \ref{recovery}(b).
For the first inequality in \eqref{Bcontract} we have used the invariance (F2) of the stable foliation $\mathcal F^s$ and the contraction in (F3) for the iterates of $z$. This argument is 
justified by the next claim.
Recall that $U$ is the domain where $e^s$ makes sense (See Sect.\ref{critical}).
\begin{claim}
For every $1\leq i\leq p-1$, $f^{i}z\in U$.
\end{claim}
\begin{proof}

The inclusion for $i=1$ holds provided $\delta$ is sufficiently small. 
Let $i\geq2$.
Since $f^iz\in R$,  $f^iz$ is at the right of $W^s_{\rm loc}(Q)$.
On the other hand, since $\zeta\in S$,
$f^i\zeta\in W^s_{\rm loc}(Q)$ or else it is at the left of $W^s_{\rm loc}(Q)$. 
Lemma \ref{dist}(b) implies $|f^i\zeta-f^iz|\leq 21\tau$.  
Let $\ell$ denote the straight segment connecting $f^iz$ and $f^i\zeta$. Let $y$ denote the point of intersection
between $\ell$ and $W^s_{\rm loc}(Q)$. Since $b\ll1$, $f^iz$ and $f^i\zeta$ are near the $x$-axis, and so $y\in B(\tau)$ holds.
Hence $f^iz\in B(22\tau)\subset U.$
\end{proof}
(\ref{Aexpand}) (\ref{Bcontract}) yield
$$\|Df^iv(z)\|\approx|A|\cdot\Vert Df^{i-1}(fz)\left(\begin{smallmatrix}
1\\
0\end{smallmatrix}\right)\Vert\approx |\zeta-z| \cdot \|w_{i}(\zeta)\|,$$
and hence (c).

\medskip

Let $i\leq q$.
The definition of $q$ and $\|w_i(\zeta)\|\leq \|w_q(\zeta)\|$ give
$$|A|\cdot\Vert Df^{i-1}(fz)\Vert\leq |\zeta-z| \cdot \|w_{i}(\zeta)\|\leq
 |\zeta-z| \cdot \|w_{q}(\zeta)\|\leq  |\zeta-z|^{1-\beta}\ll1.$$ This and 
 $ |B|\cdot\Vert Df^{i-1}e^s(fz)\Vert\leq
(Cb)^{i-1}$
yield (d).

As for (e) \eqref{squaredistance} gives $|\zeta-z|^{-1}\geq C\lambda_1^{\frac{p}{2}}$.
\eqref{square0} and the first inequality of 
Lemma \ref{size2}(b) give
$\|w_p(\zeta)\|\cdot |\zeta-z|^2\geq
C\|w_p(\zeta)\|\cdot D_p(\zeta)\geq C\tau.$ Hence
$$\Vert Df^pv(z)\Vert\geq
 C\|w_{p}(\zeta)\|\cdot|\zeta-z|\geq
C\tau|\zeta-z|^{-1}\geq C\tau\lambda_1^{\frac{p}{2}}\geq(4-\varepsilon)^{\frac{p}{2}},$$
where the last inequality holds provided $\delta$ is sufficiently small.
(f) follows from (c).
%$\frac{\|Df^jv(z)\|}{\|Df^jv(z)\|}\approx
%\frac{\|w_j(\zeta)\|}{\|w_i(\zeta)\|}$ for all $q\leq i<j\leq p$, which in turn% follows from (c).
\qed
\medskip

\subsection{Unstable leaves}\label{uleaf}
In order to use Proposition \ref{recovery} for a global analysis of the dynamics on $\Omega$,
 we have to find critical points in a suitable position for each return to $I(\delta)$.
To this end  we show that part of $\Omega$ is contained in the union of one-dimensional leaves, which are accumulated by sufficiently long $C^2(b)$-curves
in $W^u$.

Let
$\tilde\Gamma^u$ denote the collection of $C^2(b)$-curves in $W^u$ with endpoints in the stable sides of $\Theta$. Let
$$\Gamma^u=\{\gamma^u\colon\text{$\gamma^u$ is the pointwise limit of
a sequence in $\tilde\Gamma^u$}\}.$$
Any curve in $\Gamma^u$ is called an {\it unstable leaf}.
By the $C^2(b)$-property, the pointwise
convergence is equivalent to the uniform convergence. Since two
distinct curves in $\tilde\Gamma^u$ do not intersect each other,
the uniform convergence is equivalent to the $C^1$ convergence.
Hence, any unstable leaf is a $C^1$ curve with endpoints in the stable sides of $\Theta$
and the slopes of its
tangent directions are $\leq\sqrt{b}$.
Let $\mathcal W^u$ denote the union of all unstable leaves.
\begin{lemma}\label{incl}
$\Theta\cap \Omega\subset \mathcal W^u$.
\end{lemma}
\begin{proof}
Let $z\in \Theta\cap \Omega$.
Then there exists an arbitrarily large integer $k$ such that
$f^{-k}z\notin I(\delta)$. Since $z\in \Omega$, $f^{-k}z\in R$. Hence, $z\in\Delta_k$ holds. Since $k$ can be made arbitrarily large, from Lemma \ref{geo} $z$ is accumulated by curves in
$\tilde\Gamma^u$. Hence $z$ is contained in an unstable leaf.
\end{proof}

\subsection{Bound/free structure}\label{bf}
Let $z\in \Omega\cap I(\delta)$.
To the forward orbit of $z$ we associate inductively a sequence of integers
$0=:n_0<n_0+p_0<
n_1<n_1+p_1<n_2<n_2+p_2<\cdots$,
and then introduce useful terminologies along the way.

\begin{lemma}\label{tangent}
If $z\in \Omega\cap I(\delta)$, then there exists a critical point $\zeta$ and
a $C^2(b)$-curve $\gamma$ which contains $z$, $\zeta$ and is tangent to $E^u(z)$, $E^u(\zeta)$.
\end{lemma}
\begin{proof}
Since $z\in \Omega\cap I(\delta)$, by Lemma \ref{incl} it is accumulated by $C^2(b)$-curves
in $W^u$ with endpoints in the stable sides of $\Theta$, each of which
admits a critical points. Hence the claim follows.
\end{proof}

%\begin{lemma}\label{emit}
%Let $\zeta$ be a critical point non escaping up to time $n$ and $f^{n+1}\zeta\notin U$. If $z\in I(\delta)$ and $p(z)>n$, then $f^{n+1}z\notin R$. 
%\end{lemma}
%\begin{proof}
%By Lemma \ref{dist}(b) we have $|f^nz-f^n\zeta|\leq 20\tau+(Cb)^{n-1}$,
%and thus $|f^{n+1}z-f^{n+1}\zeta|\leq 5(\tau+(Cb)^{n-1}).$
%Since $\tau\ll1$, this implies $f^{n+1}z\notin R$. \end{proof}

Given $n_i$ with $f^{n_i}z\in I(\delta)$,
in view of Lemma \ref{tangent} take a critical point $\zeta$ and a $C^2(b)$-curve $\gamma$
in $I(\delta)$
which contains $f^{n_i}z$, $\zeta$ and is tangent to $E^u(f^{n_i}z)$, $E^u(\zeta)$.
Let
$p_i=p(f^{n_i}z)$ denote the bound period of $f^{n_i}z$ given by the definition in
Sect.\ref{recover} applied to $(\zeta,\gamma)$.

We claim that $p_i$ makes sense. This is clear if $\zeta=\zeta_0$.
Consider the case $\zeta\neq\zeta_0$. Then $n(\zeta)<+\infty$.
If $p_i$ does not make sense, then $f^{n_i}z$ comes too close to $\zeta$, so that
$|\zeta-f^{n_i}z|\leq C\cdot D_{n(\zeta)}(\zeta)$ for some $C>0$.
The estimate in Lemma \ref{dist}(b) implies 
$|f^{n(\zeta)+1}z-f^{n(\zeta)+1}\zeta|\leq 21\tau.$
Since $f^{n(\zeta)+1}\zeta\notin U$ and $B(22\tau)\subset U$,
$f^{n(\zeta)+1}z\notin R$ holds. This yields a contradiction to the assumption that $z\in\Omega$.
Hence the claim follows.

Let $n_{i+1}$ denote the next return time of the orbit of $z$ to $I(\delta)$
after $n_i+p_i$. Then Lemma \ref{tangent} applies to $f^{n_{i+1}}z$.
A recursive argument allows us to decompose the forward orbit of $z$ into segments corresponding to time intervals $(n_i,n_i+p_i)$ and $[n_i+p_i,n_{i+1}]$, during which we describe the points in the orbit of $z$ as being ``bound" and ``free" states respectively. Each $n_i$ is called a {\it free return time.}

Let us record the following derivative estimates: \begin{equation}\label{delest}
\|Df^{p_i}|E^u(f^{n_i}z)\|\geq(4-\varepsilon)^{\frac{p_i}{2}}\ \ \text{and}\ \
\|Df^{n_{i+1}-n_i-p_i}|E^u(f^{n_i+p_i}z)\|\geq\sigma^{n_{i+1}-n_i-p_i}.\end{equation}
The first one is a consequence of Proposition \ref{recovery}.
The second one follows from Lemma \ref{outside} and the fact that 
$E^u(f^{n_i+p_i}z)$ is spanned by a $b$-horizontal vector, which in turn follows
from Proposition \ref{recovery}(f).

\section{Symbolic coding}
In this section we show that $f|\Omega$ is semi-conjugate to the full shift on two symbols. As a corollary we obtain an upper semi-continuity of entropy.
In Sect.\ref{US} we give precise statements of main results in this section.
In Sect.\ref{relevant} we introduce some relevant definitions, and in Sect.\ref{pfprop} we construct the semi-conjugacy.
\subsection{Upper semi-continuity of entropy}\label{US}
The region $R\setminus {\rm int}S$ consists of two
rectangles, intersecting each other only at $\zeta_0$. Let $R_0$ denote
the one at the left of $\zeta_0$ and let $R_1$ denote the one at
the right.
Let $\Sigma_2=\{0,1\}^{\mathbb Z}$ denote the shift space endowed with the product topology of the discrete topology in $\{0,1\}$. 
Let 
$$K=\{z\in\mathbb R^2\colon\{f^nz\}_{n\in\mathbb Z}\text{ is bounded}\}.$$
Since any point outside of $R$ goes to infinity under positive or negative iteration,
$K=\bigcap_{n\in\mathbb Z}f^nR$.
Let $\pi\colon \Sigma_2\to K$
denote the coding map, namely, for $\omega=(\omega_n)_{n\in\mathbb
Z}\in\Sigma_2$ let
$$\pi(\omega)=\{x\in K\colon f^nx\in R_{\omega_n}\ \forall
n\in\mathbb Z\}.$$
Let $\sigma\colon\Sigma_2\circlearrowleft$ denote the left shift.

\begin{prop}\label{iota}
For any $\omega\in\Sigma_2$, $\pi(\omega)$ is a singleton. In addition, $\pi$ is surjective, continuous, 1-1 except on
$\bigcup_{i=-\infty}^\infty f^{i}\zeta_0$ where it is 2-1. It
gives a semi-conjugacy $\pi\circ \sigma=f\circ\pi$.
\end{prop}
It follows that any point in $K$ is non-wandering, and thus $K\subset\Omega$.
Since $\Omega$ is bounded, $\Omega\subset K$.
Hence we obtain $K=\Omega$, and the next

\begin{cor}\label{USC}
The entropy map
$\mu\in\mathcal M(f)\mapsto h(\mu)$ is upper semi-continuous.
In particular,
there exists an equilibrium measure for any continuous potential. Moreover, 
for a dense set of continuous potentials this equilibrium measure is unique. 
\end{cor}

\begin{proof}
Let $\mathcal M(\sigma)$ denote the space of $\sigma$-invariant Borel probability measures endowed with the topology of weak convergence. The push-forward
$\pi_*\colon \mathcal M(\sigma)\to\mathcal M(f)$ is a continuous map from a compact space to a Hausdorff space.
To show that $\pi_*$ is bijective, we use the following, the proof of which is left as an exercise.
\begin{claim}\label{iso}
Let $X_i$ be a topological space and $\mathcal B_i$ its Borel $\sigma$-algebra, $i=1,2$.
Let $h\colon X_1\to X_2$ be a bijective map which sends open sets
to Borel sets. Then $h^{-1}$ is measurable.
\end{claim}

%\begin{proof}
%Let $Z_2=\{A_2\subset X_2\colon \exists A_1\in\mathcal B_1\ \text{s.t.}\ hA_1=A%_2\}$. It suffices to show $Z_2\subset\mathcal B_2$.
%Let $Z_1=\{A_1\subset X_1\colon \exists A_2\in\mathcal B_2\ \text{s.t.}\ h^{-1}%A_2=A_1\}$.
%The assumptions on $h$ imply that  $Z_1$ is a $\sigma$-algebra containing any o%pen set. Hence $Z_1\supset \mathcal B_1$ holds.
%Suppose $Z_2\nsubseteq\mathcal B_2,$ namely,
%there exist $A_1\in\mathcal B_1$, $A_2\subset X_2$, $A_2\notin\mathcal B_2$ suc%h that $A_2=hA_1$. If $A_1\in Z_1$, then $A_2\in\mathcal B_2$ holds.
%Hence $A_1\notin Z_1$, which yields a contradiction to the inclusion  $Z_1\sups%et \mathcal B_1$. Hence $Z_2\subset\mathcal B_2$ holds.
%\end{proof}

Let $K_0=\pi^{-1}\left(K\setminus\bigcup_{n=-\infty}^{\infty}f^{n}\zeta_0\right)$ and $\pi_0=\pi|K_0$.
Since $\pi_0$ is bijective and sends open sets to measurable sets, by Claim~\ref{iso} it is a measurable bijection, and thus the pull-back $\pi_0^*$ is well-defined.
Since $\zeta_0$ is not a periodic point, any $\nu\in\mathcal M(f)$ gives full weight to $K_0$,
and so $\pi_0^*(\nu)\in\mathcal M(\sigma)$.
Hence $\pi_*$ is bijective. In particular $\pi_*$ is a homeomorphism, and the inverse is $\pi_0^*$.
Then the existence of equilibrium measures for any continuous potential follows directly from
the upper semi-continuity of the entropy map of $\sigma$.
The uniqueness follows from \cite[Corollary 9.15.1]{Wal82}.
\end{proof}

\begin{figure}
\begin{center}
\includegraphics[height=4cm,width=8.3cm]{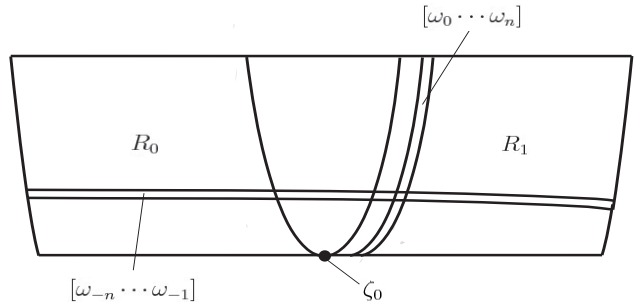}
\caption{The regions $R_0$, $R_1$ and the $s/u$-rectangles}
\end{center}
\end{figure}

\subsection{$s/u$-rectangles}\label{relevant}
By an {\it $s$-rectangle} we mean a rectangle in $R$ whose
unstable sides belong to the unstable sides of $R$. A {\it
$u$-rectangle} is a rectangle in $R$ whose stable sides belong to
the stable sides of $R$.
Let $\omega=\{\omega_n\}_{n\in\mathbb Z}\in\Sigma_2$ and
write
$\omega=\{\omega^-,\omega^+\}\in\Sigma_2$, where
$\omega^-=\{\omega_n\}_{n<0}$ and $\omega^+=\{\omega_n\}_{n\geq0}$.
For $k\leq l$, let
$$[\omega_k,\omega_{k+1},\ldots,\omega_{l}]=\bigcap_{k\leq n\leq l} f^{-n}(R_{\omega_n}).$$
 If $k\geq0$, then this set is an $s$-rectangle in $R_{\omega_0}$.
 if $l\leq-1$, then it is a  $u$-rectangle.
Set
$V^u(\omega^-)=\bigcap_{n<0}[\omega_{-n}\cdots \omega_{-1}]$
and $V^s(\omega^+)=\bigcap_{n\geq0}[\omega_0\cdots \omega_n]$.
We have $\pi(\omega)=V^u(\omega^-)\cap V^s(\omega^+)$.

%By a {\it $u$-curve} we mean the unstable sides of $\Theta$, or else a
%compact $C^1$ curve in $\Theta$ which connects the stable sides
%of $\Theta$, and intersects the boundary of $\Theta$ at no point
%other than its endpoints. %An {\it $s$-curve} is defined similarly.
%For an $u$-curve $\gamma$, 

\subsection{Proof of Proposition \ref{iota}}\label{pfprop}
Let $\omega\in\Sigma_2$. We show that
$\pi(\omega)$ is a singleton.
In the coding of the uniformly hyperbolic horseshoe, one considers
families of stable and unstable strips ($s/u$-rectangles in
our terms) and show that their boundary curves converge to curves,
intersecting each other exactly at one point. In our situation,
due to the presence of tangency, the convergence of the stable sides of
$s$-rectangles is not clear. To circumvent this
point, we take advantage of the fact that $f=f_{a^*}$ and $a^*$ is the first bifurcation parameter.

\begin{lemma}\label{single}
If $\Theta\cap\pi(\omega)\neq\emptyset$, then $\pi(\omega)$ is a singleton.
\end{lemma}

\begin{proof}

%Lemma \ref{geo} implies that $\Theta\cap V^u(\omega^-)$ is a
%$C^1$ limit of the unstable sides of
%$\pi[\omega_{-n}\cdots\omega_{-1}]$, and hence is an $u$-curve.
%Meanwhile,

For each $n>0$
let $\partial^s[\omega_0\cdots\omega_n]$ denote any stable side of the $s$-rectangle $[\omega_0\cdots\omega_n]$.
Since $W^s(P)$ does not
intersect itself, either
$\partial^s[\omega_0\cdots\omega_n]\subset \Theta$ or
$\subset R\setminus\Theta$. 
The next sublemma implies that
if $\partial^s[\omega_0\cdots\omega_n]\subset\Theta$, then it does not wind around $\Theta\cap V^u(\omega^-)$ (See Figure 5).
For $\gamma\in\mathcal W^u$,
let
$D(\gamma)$ denote the closed domain bordered by $\gamma$,
the unstable side of $\Theta$ containing $\zeta_0$ and the stable
sides of $\Theta$. If $\gamma$ is one of the unstable sides of
$\Theta$, then let $D(\gamma)=\Theta$.

\begin{sublemma}\label{lem3}
If $\partial^s[\omega_0\cdots\omega_n]\subset\Theta$, then
$\partial^s[\omega_0\cdots\omega_n]\cap D(\Theta\cap V^u(\omega^-))$
is connected.
\end{sublemma}
\begin{proof}
Suppose this intersection is not connected. By Lemma \ref{geo},
the unstable sides of $\Theta\cap[\omega_{-n}\cdots\omega_{-1}]$
are $C^2(b)$-curves, and
converge in $C^1$ to the curve $\Theta\cap V^u(\omega^-)\in W^u$. Hence, it is
possible to choose an integer $m>0$ and an unstable side
$\gamma$ of $\Theta\cap[\omega_{-m}\cdots\omega_{-1}]$ such that
$\partial^s[\omega_0\cdots\omega_n]\cap
D(\gamma)$ is not
connected.

Since the endpoints of $\gamma$ and 
$\partial^s[\omega_0\cdots\omega_n]$
are transverse homoclinic or heteroclinic points, and the transversality persists under small modifications 
of the parameter, for $a$ bigger than and close to $a^*$ one can consider the \emph{continuations} $\gamma(a)$, 
$\partial^s[\omega_0\cdots\omega_n](a)$ of these two curves.
For the same reason, the domain $D(\cdot)$ makes sense for $f_a$.
Since $a^*$ is the first bifurcation parameter,
$f_{a}$ for $a>a^*$ is Smale's horseshoe map. Hence,
$\partial^s[\omega_0\cdots\omega_n](a)\cap
D(\gamma(a))$ has to be
connected. By the continuous parameter
dependence of invariant manifolds, there must come a parameter $a_0>a^*$ such that $\partial^s[\omega_0\cdots\omega_n](a_0)$ meets
$\gamma(a_0)$ 
tangentially. This yields a contradiction to the fact that $a^*$ is the first bifurcation parameter.
\end{proof}

Since $\Theta\cap\pi(\omega)\neq\emptyset$, at least one of the stable sides of $[\omega_0\cdots\omega_n]$
is contained in $\Theta$, and so intersects $\Theta\cap V^u(\omega^-)$.
By Sublemma \ref{lem3}, $\{\Theta\cap V^u(\omega^-)\cap
[\omega_0\cdots\omega_n]\}_{n>0}$ is a strictly decreasing sequence of closed curves in $\Theta\cap
V^u(\omega^-)$. Hence, $\Theta\cap
V^u(\omega^-)\cap V^s(\omega^+)=\Theta\cap\pi(\omega)$ is a singleton, or else
a closed curve. We argue by contradiction to eliminate the latter alternative.

Suppose that $\gamma:=\Theta\cap\pi(\omega)$ is not a singleton. Then it is a closed curve. Since $\gamma$ is $C^1$ 
accumulated by curves in $\tilde\Gamma^u$, one can define a bound/free
structure for any point in $\gamma$.
Suppose that $x,y\in\gamma$, $n>0$ are such that $f^nx$ is bound and $f^ny\in I(\delta)$. Then $f^nx$ is near $Q$, and thus $f^{n+1}\gamma$ intersects both
$R_0$ and $R_1$. This yields a contradiction. Hence, it follows that if $x\in\gamma$, $n>0$ and $f^nx$ is bound, then $f^n\gamma\cap I(\delta)=\emptyset$.
Then one can take an arbitrarily large integer $n$ such that all points on $f^n\gamma$ are free. Proposition \ref{recovery}
yields ${\rm length}(f^n\gamma)\geq \delta(4-\varepsilon)^{\frac{n}{3}}\cdot{\rm length}(\gamma)$, and that the tangent vectors
of $\gamma$ are $b$-horizontal.  Hence, some forward iterates of $\gamma$ intersect both $R_0$ and $R_1$, a contradiction. This completes the proof of Lemma \ref{single}.
\end{proof}

\begin{figure}
\begin{center}
\includegraphics[height=3cm,width=6cm]{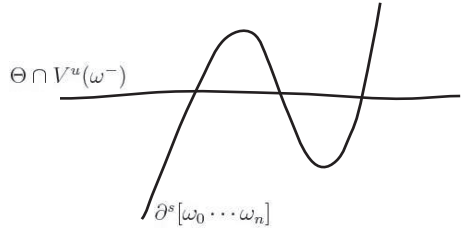}
\caption{The situation  eliminated by Sublemma \ref{lem3}  in which
 $\partial^s[\omega_0\cdots\omega_n]\cap D(\Theta\cap V^u(\omega^-))$ is not connected.}
\end{center}
\end{figure}

For $n\in\mathbb Z$, let $A_n(\omega)=\{x\in\pi(\omega)\colon
f^nx\in\Theta\}.$

\begin{lemma}\label{single2}
The following holds for all $m$, $n\in\mathbb Z$:

\begin{itemize}
\item[(a)] $A_n(\omega)$ is a singleton unless it is empty;

\item[(b)] either (i) $A_m(\omega)=A_n(\omega)$, or (ii) 
$A_m(\omega)=\emptyset$ or $A_n(\omega)=\emptyset$.
\end{itemize}
\end{lemma}
\begin{proof}
We have $ f^nA_n(\omega)=\Theta\cap\pi(\sigma^n\omega)$. Hence
Lemma \ref{single} gives (a). To prove (b)
we need 
\begin{sublemma}\label{lem}
If $x\in R\setminus\Theta$ and $y\in\Theta$, then
$\pi^{-1}(x)\cap\pi^{-1}(y)=\emptyset$.
\end{sublemma}

We finish the proof of Lemma \ref{single2}(b) assuming Sublemma \ref{lem}.
If (i) (ii) do not hold,
then $f^mA_m(\omega)\subset\Theta$ and 
$f^mA_n(\omega)\subset R\setminus\Theta$. 
We have $\pi^{-1}(f^mA_m(\omega))=\pi^{-1}(f^mA_n(\omega))$, while
Sublemma \ref{lem} gives $\pi^{-1}(f^mA_m(\omega))\cap\pi^{-1}(f^mA_n(\omega))=\emptyset.$
This yields a contradiction.

It is left to prove Sublemma \ref{lem}.
For $x\in K$ and
$n\in\mathbb Z$, define $\omega_n(x) \in\{0,1\}$ by
 $f^nx\in R_{\omega_n(x)}$. In the case $f^nx=\zeta_0$ we let $\omega_n(x)=0$ or $1$.
It suffices to claim that if $x\in R\setminus\Theta$ and $y\in\Theta$, then
there exists
$n\geq0$ such that $\omega_n(x)\neq \omega_n(y)$.
To see this, define rectangles $S_1,S_2,S_3,S_4$ as follows:
$S_1$ (resp. $S_4$) is the component of $R\setminus{\rm Int} \Theta$ at the left (resp. right) of $\zeta_0$; $S_2=R_0\setminus {\rm int}S_1$ and $S_3=R_1\setminus {\rm int}S_4$ (See Figure 6).
Observe that: $fS_1\subset S_1\cup S_2\cup S_3$; $fS_2\subset S_4$; $fS_3\subset S_4$; $fS_4\subset S_1\cup S_2\cup S_3$.
Either: (i) $x\in S_1$, $y\in S_2$; (ii) $x\in S_1$, $y\in S_3$; (iii) $x\in S_4$, $y\in S_2$; (iv) $x\in S_4$, $y\in S_3$.
In cases (ii) and (iii) we have $\omega_0(x)\neq \omega_0(y)$, and so the claim holds with $n=0$.
In case (i) either $\omega_0(x)\omega_1(x)=00$, $\omega_0(y)\omega_1(y)=01$ and the claim holds with $n=1$, or else $fx\in S_3$,
$fy\in S_4$ which is reduced case (iv).

We now consider case (iv). Then either $\omega_0(x)\omega_1(x)=10$, $\omega_0(y)\omega_1(y)=11$
and the claim holds with $n=2$, or else $fx\in S_3$, $fy\in S_4$
and $\omega_0(x)\omega_1(x)=11=\omega_0(y)\omega_1(y).$
If $fx\in I(\delta)$ then $\omega_0(x)\omega_1(x)\omega_2(x)\omega_3(x)=1110$,
$\omega_0(y)\omega_1(y)\omega_2(y)\omega_3(y)\in\{1100, 1101,1111\}$. 
Hence the claim holds with $n=2$ or $3$. 

Let us now assume that $fx\notin I(\delta)$. 
Let $z$ denote the point of intersection between $V^s(\omega^+)$ and the unstable leaf containing $x$.
Let $L$ denote the segment connecting $z$ and $x$.
Lemma~\ref{outside} implies that the lengths of the forward images 
of $L$ grow exponentially as long as 
the images does not meet $I(\delta)$. Let $k>1$ be the smallest positive integer such that $I(\delta)\cap f^kL\neq\emptyset$.
Since $f\alpha_1^+\subset \alpha_1^+$ and $L$ intersects $\alpha_1^+$,
$\omega_i(x)=1=\omega_i(y)$ for $0\leq i\leq k-1$.
If $f^{k}x\in I(\delta)$ then $f^{k}y\in S_4$, and so 
$\omega_{k}(y)\omega_{k+1}(y)\omega_{k+2}(y)\in\{100,101,111\}.$
As for $x$,
$\omega_{k}(x)=0$, or else  $\omega_{k}(x)\omega_{k+1}(x)\omega_{k+1}(x)=110$.
Hence the claim holds with
$n=k,k+1$ or $k+2$.
The same reasoning holds for the case $f^{k}y\in I(\delta)$.
\end{proof}

We are in position to complete the proof of Proposition \ref{iota}.
Let  $E=\{x\in K\colon f^nz\notin\Theta\ \forall n\in\mathbb Z\}.$
%It is easy to see that the set of points in $R$ whose
%forward orbits do not intersect $\Theta$ coincides with the stable
%sides of $R$. Due to the symmetric reason, the set of points in
%$R$ whose backward orbits do not intersect $\Theta$ coincides with
%the unstable sides of $R$.
We have
\begin{equation}\label{pio}
\pi(\omega)=\{x\in E\colon \pi^{-1}(x)=\omega\}\cup\bigcup_{n\in\mathbb Z} A_n(\omega).
\end{equation}

It is easy to see that $E$ is contained in the stable sides of $R$.
In addition, Lemma \ref{geo} implies that
if $x,y\in K$ belong to the same stable side of $R$, then $\pi^{-1}(x)\neq\pi^{-1}(y)$. Hence the first set in (\ref{pio})
is a singleton unless it is empty. By Lemma \ref{single2}, the second set in (\ref{pio}) is a singleton unless
it is empty. Either the first or the second set is empty, for otherwise
Sublemma \ref{lem} yields a contradiction. Consequently, $\pi(\omega)$ is a sigleton. 

Since $R_0\cap
R_1=\{\zeta_0\}$, $\pi$ is 1-1 except on
$\bigcup_{i=-\infty}^\infty f^i\zeta_0$ where it is 2-1. Observe
that, since $\sigma$ sends cylinder sets to cylinder sets, the
continuity of $\pi$ at a point $\omega$ implies the continuity of $\pi$ at
$\sigma^n\omega$, $n\in\mathbb Z$. The continuity of $\pi$ on $\pi^{-1}\Theta$ follows from the proof of Lemma
\ref{single}. By the above observation, $\pi$ is continuous on $\Sigma_2\setminus\pi^{-1}E$. The continuity on $\pi^{-1}E$ is obvious. Since $K\subset R_0\cup R_1$, $\pi$ is surjective.
\qed

\section{Proof of the theorem}
In this section we finish the proof of the theorem. In Sect.\ref{unstable} we study the regularity of the unstable direction $E^u$ defined in (\ref{eu}).
In Sect.\ref{limpts} we estimate the amount of drop of unstable Lyapunov exponents 
in the weak convergence of measures.  
In Sect.\ref{proff} we prove the theorem.

\subsection{Regularity of the unstable direction}\label{unstable}
We first show that $E^u$ is Borel measurable. For two positive integers $i,j$, $j>1$, let $\Omega_{i,j}$ denote the set of
all $z\in \Omega$ for which there exists $v\in T_z\mathbb R^2\setminus\{0\}$ such
that $\|Df^{-n}(z)v\| \leq ij^{-n}\|v\|$ holds for every $n\geq0.$
Clearly, $\Omega_{i,j}$ is a closed set.
Observe that $E^u(z)$ makes sense if and only if there exist $i,j$ such that
$z\in \Omega_{i,j}$. Since $E^u$ is continuous on
 $\Omega_{i,j}$, it is Borel measurable on $\bigcup_{i,j} \Omega_{i,j}$.

Due to the
presence of the tangency, $E^u$ is not continuous at $Q$.
We show that $E^u$ makes sense, and is continuous
on a large subset \footnote{We do not make any claim on the continuity of $E^u$ on $\partial^sR\setminus\{Q\}$.
This does not matter because $f$-invariant probability measures do not charge this set.}
of $\Omega$.
Let $\partial^sR$ denote the union of the stable sides of $R$ and let
$\Omega'=\Omega\setminus\partial^sR$.

\begin{prop}\label{measurable}
$E^u$ is well-defined on $\Omega$, and is continuous on $\Omega'$.
\end{prop}

\begin{figure}
\begin{center}
\includegraphics[height=4cm,width=8.3cm]{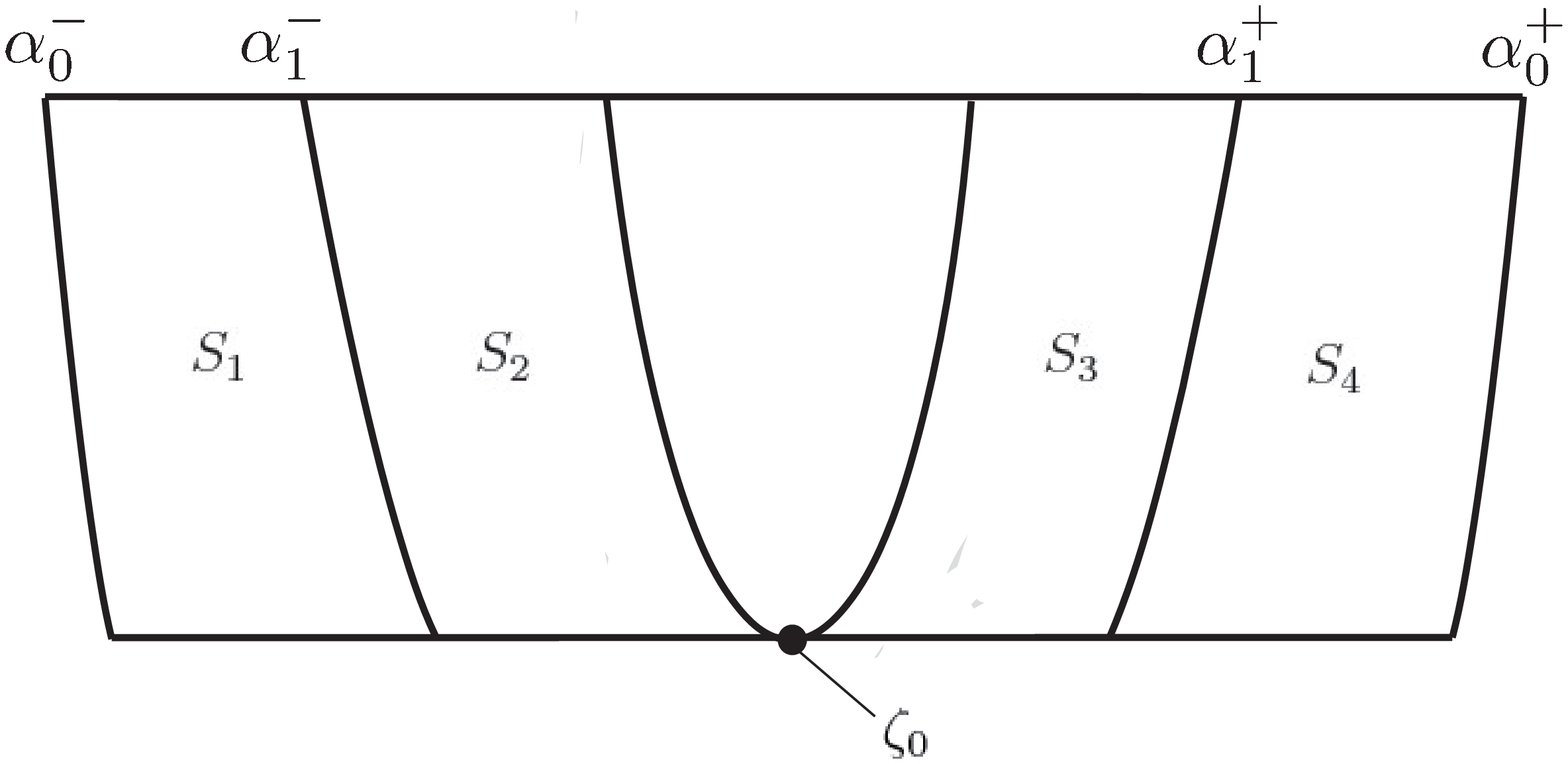}
\caption{The region $S_1$, $S_2$, $S_3$, $S_4$}
\end{center}
\end{figure}

\begin{proof}
We first prove that $E^u$ makes sense on $\mathcal W^u$,
and is spanned by the tangent directions of the unstable leaves in $\Gamma^u$.
Since any unstable leaf is a $C^1$ limit of a sequence of
curves in $\tilde\Gamma^u$, these statements follow from
the next uniform backward contraction on curves in $\tilde\Gamma^u$.
\begin{lemma}\label{contract}
There exists $C>0$ such that for any $\gamma\in \tilde\Gamma^u$, $z\in\gamma$ and $n>0$,
$\|Df^n|E^u(f^{-n}z)\|\geq  C(4-\varepsilon)^{\frac{n}{2}}.$ %\frac{1}{10}\lambda_1^{\frac{n}{2}}.$
\end{lemma}
\begin{proof}
Take a large integer $M\geq n$ so that $f^{-M}z$ is contained in the local unstable manifold of the saddle.
We introduce a bound/free structure for the forward orbit of $f^{-M}z$.
Observe that $z\in\Theta$ must be free, as
the forward orbit of a critical point never returns close to $\Theta$.

We first consider the case where
$f^{-n}z$ is free. Splitting the orbit $f^{-n}z,f^{-n+1}z\ldots,z$
into bound and free segments, and then applying the derivative estimates in 
(\ref{delest}) we get the desired inequality.

We now consider the case where
$f^{-n}z$ is bound. Let $i$ denote the smallest $j>n$ such that
$f^{-j}z\in I(\delta)$.
Let $p$, $q$ denote the corresponding bound and fold periods. We have $-n<-i+p$. There are two cases,
$-n$ being either inside or outside of the fold period.
If $-n<-i+q$, then
\[\|Df^n|E^u(f^{-n}z)\|=\frac{\|Df^i|E^u(f^{-i}z)\|}{\|Df^{i-n}|E^u(f^{-i}z)\|}
\geq \|Df^i|E^u(f^{-i}z)\|\geq (4-\varepsilon)^{\frac{i}{2}}>(4-\varepsilon)^{\frac{n}{2}}.\]
For the first inequality we have used Proposition \ref{recovery} (d).
If $-n\geq-i+q$, then by Proposition \ref{recovery} (c) and \eqref{misiurewicz} for some $C\in(0,1)$ we have
\[\frac{\|Df^{p}|E^u(f^{-i}z)\|}
{\|Df^{i-n}|E^u(f^{-i}z)\|}\geq C\frac{\|w_p(\zeta)\|}
{\|w_{i-n}(\zeta)\|} \geq C\lambda_1^{p-i+n}.\]
Since both
$f^{p-i}z$ and $f^{-i}z$ are free,
Proposition \ref{recovery} (e) and Lemma \ref{outside} yield
\[\|Df^i|E^u(f^{-i}z)\|\geq (4-\varepsilon)^{\frac{i-p}{2}}\|Df^{p}|E^u(f^{-i}z)\|.\] Multiplying these two
inequalities and then using $p-i+n>0$, $\lambda_1>(4-\varepsilon)^{\frac{1}{2}}>1$ we obtain 
\[\|Df^n|E^u(f^{-n}z)\|\geq C\lambda_1^{p-i+n}(4-\varepsilon)^{\frac{i-p}{2}}
\geq C(4-\varepsilon)^{\frac{n}{2}}.\qedhere\]
\end{proof}
Lemma \ref{contract} shows that $E^u$ is well-defined on $\mathcal W^u$. By Lemma \ref{incl}, $E^u$ is defined everywhere on
$\Theta\cap \Omega$. Due to the $Df$-invariance,
$E^u$ is well-defined everywhere on $\bigcup_{n=-\infty}^{+\infty}
f^n(\Theta\cap \Omega)$.
If $z\in\Omega$ is not contained in 
$\bigcup_{n=-\infty}^{+\infty} f^n(\Theta\cap \Omega)$, then $z\in\partial^sR$.
Since the dynamics outside of $\Theta$ is uniformly hyperbolic, the standard cone field argument
shows that $E^u(z)$ is well-defined.
Therefore, $E^u$ is well-defined on $\Omega$.

%We now prove Sublemma \ref{incl}.
%Let $z\in\Theta\cap K$.
%Then there exists an arbitrarily large integer $n$ such that
%$f^{-n}z\notin\Delta$. We introduce a bound/free structure for the orbit
%$f^{-n}z,\ldots,f^{-1}z,z$ with $v=\left(\begin{smallmatrix}1\\0\end{smallmatri%x}\right)\in T_z\mathbb R^2$. We have $\|Df^nv\|\geq 10^{-20}\|Df^iv\|$ for $0\%leq i<n$.
%By \cite{WanYou01} or Pliss's Lemma \cite{}, it is possible to choose an intege%r $n/2\leq m\leq n$ such that
%$f^{-m}z\notin\Delta$ and $\|Df^{i}(f^{-m}z)\|\geq b^{\frac{i}{4}}$ for every $%0<i\leq m$. By a result of \cite{MorVia93},
%there exists a curve $\Gamma$ through $f^{-m}z$ with the following properties:
%(i) for any $x,y\in \Gamma$, $|f^mx-f^my|\leq (Cb)^{\frac{m}{2}}$;
%(ii) ${\rm length}(\Gamma)\geq 10^{-10}.$
%In particular,
%$\Gamma$ intersects $W^u$. It follows that the ball of radius $(Cb)^{\frac{m}{2%}}$
%intersects a small segment in $W^u$. By Lemma \ref{c2}, there exists a curve in% $\tilde\Gamma^u$ which contains this small segment. Since $m$ can be made arbi%trarily large by taking large $n$, $z$ is accumulated by curves in
%$\tilde\Gamma$. Hence $z$ is contained in an unstable leaf.

 %Then for an arbitrary open set $U$ in the projective space,
%$\psi^{-1}(U)\cap\Lambda_{i,j}$ is a relatively open set in
%$\Lambda_{i,j}$, and thus is a Borel set. Hence $\psi^{-1}(U)=
%\bigcup_{i,j}\psi^{-1}(U)\cap\Lambda_{i,j}$ is a Borel set.

Lemma \ref{geo} implies
that $E^u$ is uniformly continuous on $\Theta\cap \mathcal W^u$, 
and thus it
is continuous on $\mathcal W^u$.
Let $z\in \Omega'$. Then there exists $n\geq 0$ such that
$f^nz\in \Theta$.
We first consider the case where $f^nz$ is not in the stable sides of $\Theta$.
Then $E^u$ is continuous at $f^nz$, and so
the $Df$-invariance of $E^u$ and the
continuity of $Df$ together imply that $E^u$ is continuous at $z$ as well.

In the case where $f^nz$ is in the stable sides of $\Theta$, the above argument is slightly incomplete, because the continuity of $E^u$ in a neighborhood of $f^nz$ is not proved yet. However, we can prove this by slightly extending the region $\Theta$ and repeating the same arguments.
\end{proof}

\subsection{Unstable Lyapunov exponents of limit points}\label{limpts}

%By Lemma \ref{Lyap}, it suffices to consider the case $\mu\neq\delta_Q$.
%Write
%$\mu_n=u_n\delta_Q+v_n\nu_n$, $u_n,v_n\geq0$ and $\nu_n(\{Q\})=0$.
%Passing to proper subsequences if necessary, we may assume that
%$(u_n)$ and $(v_n)$ converge.
%Let $u,v$ denote the corresponding limits. Since $\mu\neq\delta_Q$, $v>0$
%holds. Then $\mu=u\delta_Q+v\nu$, where $\nu$ is the limit point of $(\nu_n)$. %If
%$\nu(\{Q\})>0$, then $\mu=\delta_Q$. Hence $\nu(\{Q\})=0$.
%Then Lemma \ref{lyap} gives
%$\lambda_1^u(\nu_n)\to\lambda^u(\nu)$. Hence $\lambda_1^u(\mu_n)=u_n\lambda^u(
%\delta_Q)+v_n\lambda^u(\nu_n)\to u\lambda^u(\delta_Q)+v\lambda^u(\nu)=\lambda_1%^u(\mu).$\medskip

A main result in this subsection is as follows.
Let $\mathcal M^e(f)$ denote the set of all ergodic $f$-invariant Borel 
probability measures and let $\delta_Q$ denote the Dirac measure at $Q$.
\begin{prop}\label{lyap2}
If $\{\mu_n\}_n\subset\mathcal M^e(f)$,
$\mu_n\to\mu$,
$\mu= u\delta_Q+(1-u)\nu$, $0\leq u\leq1$, $\nu\in\mathcal M(f)$ and $\nu\{Q\}=0$, then:
$$ \frac{u}{2}\log(4-\varepsilon)+(1-u)\lambda^u(\nu)\leq \varliminf_{n\to\infty}\lambda^u(\mu_n);$$
$$\varlimsup_{n\to\infty}\lambda^u(\mu_n)\leq u\lambda^u(\delta_Q)+(1-u)\lambda^u(\nu).$$
\end{prop}

\begin{proof}

We first introduce a family of delimiting curves which allow us to relate the proximity of an orbit's return close to the tangency with the time it will subsequently spend near $Q$. 
Let $\tilde\alpha_0$ denote the component of $W^s(P)\cap R$ containing $P$. Define a sequence $\{\tilde\alpha_k\}_{k\geq1}$ of compact curves in $W^s(P)\cap R$ inductively as follows.
Let $\tilde\alpha_0=\alpha_1^+$ and $\tilde\alpha_0=\alpha_1^+$.
Given $\tilde\alpha_{k-1}$, $k>1$, define $\tilde\alpha_k$ to be the one of the two components of $R\cap f^{-1}\tilde\alpha_{k-1}$ which lies at the left of $\zeta_0$.
%For each $n\ge 0$ the set $f^{-2}\tilde\alpha_n\cap R$ consists of four curves, two of them at the left of $\zeta_0$ and two at the right. Let $\alpha_{n+1}^-$ denote the one to the left of $\zeta_0$ which is not $\tilde\alpha_{n+2}$. Among the two at the right of $\zeta_0$, let $\alpha_{n+1}^+$ denote the one which is at the left of the other.
The curves obey the following diagram
$$%\{\alpha_{n+1}^-,\alpha_{n+1}^+\}\stackrel{f^2}{\to}
\tilde\alpha_k
\stackrel{f}{\to}\tilde\alpha_{k-1}
\stackrel{f}{\to}\tilde\alpha_{k-2}\stackrel{f}{\to}\cdots
\stackrel{f}{\to}\tilde\alpha_1=\alpha_1^-\stackrel{f}{\to}
\tilde\alpha_0=\alpha_1^+.$$

\begin{figure}
\begin{center}
\includegraphics[height=4cm,width=8.3cm]{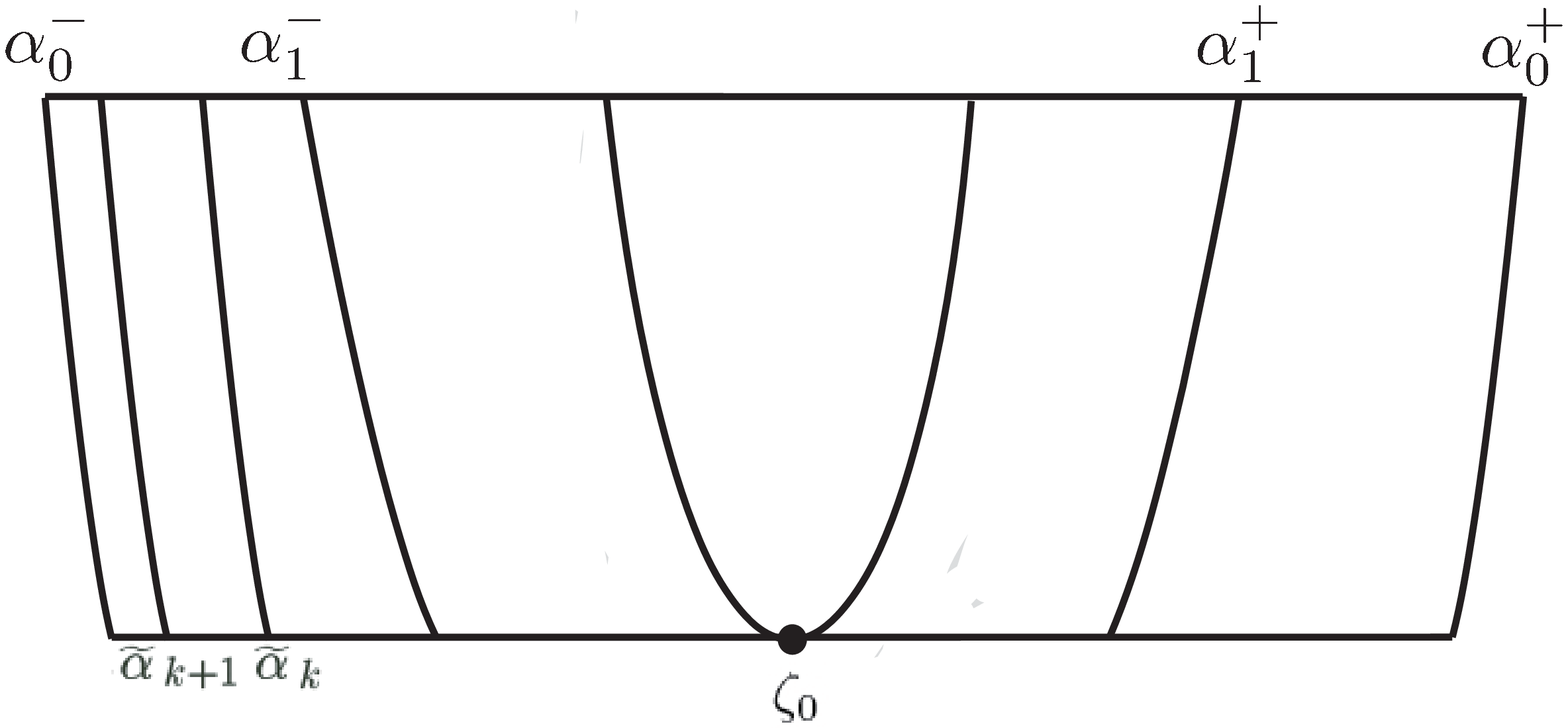}
\caption{The curves $\{\tilde\alpha_k\}$}
\end{center}
\end{figure}

For each $k\geq1$, let $\tilde V_k$ denote the rectangle containing $Q$
which is bordered by $\tilde\alpha_{k}$ and $\partial R$ (see Sect.\ref{family} for the definitions of $\tilde\alpha_{k}$). 
Let $M>0$ be a large integer, and define $V_k=V_{k,M}$ by
$$V_k=\bigcup_{i=0}^{Mk}f^i\tilde V_{2Mk}.$$
Observe that $\{V_k\}$ is a nested sequence, and $\bigcap_{k=1}^{\infty} V_k=\alpha_0^-.$

Fix a partition of unity
$\{\rho_{0,k},\rho_{1,k}\}$ on $R$ such that 
$${\rm supp}(\rho_{0,k})=\overline{\{x\in R\colon \rho_{0,k}(x)\neq0\}}\subset
V_k\ \ \text{ and }\ \ {\rm supp}(\rho_{1,k})\subset R\setminus
\overline{V_{2k}}.$$

We argue with subdivision into two cases.
\medskip

\noindent{\it Case I: $u=0$.} The desired inequalities are direct consequences of the next
\begin{lemma}\label{lyap}
If $\{\mu_n\}_n\subset\mathcal M(f)$, $\mu_n\to\mu$ and $\mu\{Q\}=0$,
then $\lambda^u(\mu_n)\to\lambda^u(\mu).$
\end{lemma}
\begin{proof}
Set $\overline
L=\displaystyle{\varlimsup_{n\to\infty}}\lambda^u(\mu_n)$ and
$\underline L=\displaystyle{\varliminf_{n\to\infty}}\lambda^u(\mu_n)$. Taking subsequences if necessary we may
assume $\overline L=\displaystyle{\lim_{n\to\infty}}\lambda^u(\mu_n)$. 
Since $\lambda^u(\mu_n)=\mu_n(\rho_{0,k}\log J^u)
+\mu_n(\rho_{1,k}\log J^u)$ and
$\rho_{1,k}\log J^u$ is
continuous by Proposition \ref{measurable},
the limit $\displaystyle{\lim_{n\to\infty}}\mu_n(\rho_{1,k}\log J^u)$ exists.
Hence, for every $k$,
\begin{equation}\label{limit}
\overline{L}=\lim_{n\to\infty}\mu_n(\rho_{0,k}\log J^u)
+\lim_{n\to\infty}\mu_n(\rho_{1,k}\log J^u).\end{equation}

Since $\mu\{Q\}=0$ and $\mu\in\mathcal M(f)$ we have $\mu(\partial V_k)=0$, and
thus $\displaystyle{\lim_{n\to\infty}}\mu_n(V_k)=\mu(V_k)$. 
Then $$\lim_{n\to\infty}\mu_n(\rho_{0,k}\log J^u)\leq
\log 5\cdot\lim_{n\to\infty}\mu_n(V_k)=\log5\cdot\mu(V_k).$$
We also have 
$\displaystyle{\lim_{k\to\infty}}\mu(V_k)=0$, and thus
the first term of the right-hand-side of (\ref{limit}) goes to $0$
as $k\to\infty$.
The weak convergence gives
$$
\lim_{n\to\infty}\mu_n(\rho_{1,k}\log J^u)= \mu(\rho_{1,k}\log
J^u).$$ From the Dominated Convergence Theorem, the second term
of the right-hand-side of (\ref{limit}) goes to $\lambda^u(\mu)$
as $k\to\infty.$ Hence we obtain $\overline L=\lambda^u(\mu)$.
The same reasoning gives $\underline L=\lambda^u(\mu)$. 
\end{proof}
\noindent{\it Case II: $u\neq0$.}
The next lemma allows us to estimate contributions of the iterates near the saddle $Q$ to the unstable Lyapunov exponents.
\begin{lemma}\label{Vk}
There exist large integers $M_0$, $k_0$ such that the following holds for all $M\geq M_0$ and $k\geq k_0$: if $z\in \Omega$, $m>0$ are such that $f^{-1}z\notin V_{k,M}$, $z,fz,\ldots,
f^{m-1}z\in V_{k,M},
f^{m}z\notin V_{k,M}$, then
\begin{equation*} \frac{1}{2}\log(4-\varepsilon)\leq\frac{1}{m}\sum_{i=0}^{m-1}\log J^u(f^iz)
\leq \lambda^u(\delta_Q).\end{equation*}
\end{lemma}

\begin{proof}
Let
$z\in \Omega$, $m>0$ be as in the statement.
We have $z\notin f^i\tilde V_{2Mk}$ for every $0<i\leq Mk$, for otherwise $f^{-1}z\in V_{k,M}$. Since $z\in V_{k,M}$, we have $z\in\tilde V_{2Mk}$. Hence
\begin{equation}\label{Vkeq2}m-1\geq Mk \ \text{ and } \ f^{-2}z\in I(\delta).\end{equation}
where the latter holds provided $k_0$ is chosen sufficiently large.

Set $y=f^{-2}z$. By Lemma \ref{tangent} there exist a critical point $\zeta$
and a $C^2(b)$-curve which contains $\zeta$, $y$, and is tangent to both $E^u(\zeta)$ and $E^u(y)$.
Let $p=p(y)$ denote the corresponding bound period.
%We first consider the case where $Df^{m+j}v(y)$ is free. Proposition \ref{recovery} gives $\|Df^{p}v(y)\|\geq\lambda^{\frac{p}{3}}$. If $k$ is sufficiently la%rge, then $f^py,f^{p+1}y,\ldots,f^{m+j-1}y$ stay close to $Q$, and thus $\|Df^{%m+j}v(y)\|\geq\lambda^{m+j-p}\|Df^pv(y)\|$. Hence we obtain
%$\|Df^{m+j}v(y)\|\geq\lambda^{\frac{m}{3}}$. Using this and (\ref{Vkeq2}),
%$$\frac{\|Df^{m+j}v(y)\|}{\|Df^jv(y)\|}\geq\lambda^{\frac{m}{3}}5^{-j}
%\geq \lambda_1^{\frac{m}{3}}5^{-\frac{m}{100}}\geq\lambda^{\frac{m}{4}}.$$
%\begin{sublemma}
%$f^2y=f^{-j+2}z\in\tilde V_{100k}\setminus\tilde V_{102k}.$
%\end{sublemma}
%\begin{proof}
%Suppose $f^2y=f^{-j+2}z\in\tilde V_{101k}.$
%Then $f^mz\in\tilde V_{101k-m-j+2}$, and thus $f^mz\in\tilde V_k$.
%On the other hand we have $f^mz\notin\tilde V_k$. This yields a contradiction.
%\end{proof}

In the sequel we argue as in the proof of Proposition \ref{recovery}.
Fix a $C^2(b)$-curve $\gamma$ which connects $fy$ and $\mathcal F^s(f\zeta)$. 
Similarly to the proof of \eqref{square0} we have ${\rm length}(\gamma)\approx|\zeta-y|^2$. Since
$f^i\gamma$ $(i=0,1,\ldots,m+1)$ 
are $C^2(b)$-curves located near the stable sides of $R$, and $f^{m+1}y=f^{m-1}z\in V_k$, $f^{m+2}y=f^mz\notin V_k$, there exists $C\geq 1$ such that 
\begin{equation}\label{Vkeq-1}C^{-1}{\lambda_2}^{-k}\leq {\rm length}(f^{m+1}\gamma)\leq C {\lambda_1}^{-k+1}.\end{equation} The bounded distortion gives 
\begin{equation}\label{Vkeq-2} |\zeta-y|^2\|w_{m+2}(\zeta)\|\approx{\rm length}(f^{m+1}\gamma).\end{equation}  
From \eqref{squaredistance} \eqref{squaredistance2}, there exists $C\geq1$ such that
\begin{equation}\label{Vkeq-3}
C^{-1}(\lambda_2+\varepsilon/2)^{-\frac{p}{2}}\leq |\zeta-y|\leq C\lambda_1^{-\frac{p}{2}}.\end{equation}

Let $v(y)$ denote any vector which spans $E^u(y)$.
 Split
$$Dfv(y)=A\cdot \left(\begin{smallmatrix}1\\0\end{smallmatrix}\right)+B\cdot e^s(fy).$$
Using \eqref{Vkeq-1} \eqref{Vkeq-2} \eqref{Vkeq-3}, for some $C>0$ we have
$$|A|\cdot\|Df^{m+1}(fy)\left(\begin{smallmatrix}1\\0\end{smallmatrix}\right)\|\approx|\zeta-y|
\cdot\|w_{m+2}(\zeta)\|\geq C{\lambda_2}^{-k}|\zeta-y|^{-1}\geq C{\lambda_2}^{-k}\lambda_1^{\frac{p}{2}}.$$
On the other hand we have
$$|B|\cdot\|Df^{m+1}e^s(fy)\|\leq (Cb)^{m+1}.$$
Since $\|Df^2v(y)\|\leq1$, if
$p\geq m+2$ then using $p>Mk$ which follows from \eqref{Vkeq2} and then choosing sufficiently large $M_0$ if necessary,
we have $$\frac{\|Df^{m+2}v(y)\|}{\|Df^2v(y)\|}\geq C {\lambda_2}^{-k}\lambda_1^{\frac{p}{2}}-(Cb)^{m+1}\geq
(4-\varepsilon)^{\frac{p}{2}}>(4-\varepsilon)^{\frac{m}{2}}.$$

Now, observe that $$\frac{1}{m}\sum_{i=0}^{m-1}\log J^u(f^iz)=\frac{1}{m}\log\frac{\|Df^{m+2}v(y)\|}{\|Df^2v(y)\|}.$$
Hence the first inequality in Lemma \ref{Vk} holds.
In the case $p< m+2$ the first inequality follows from Proposition \ref{recovery}(e)(f) and Lemma \ref{outside}.
The second inequality in the lemma is obvious.
\end{proof}

Returning to the proof of Proposition \ref{lyap2} in the case $u\neq0$, 
choose $M\geq M_0$ and $k\geq k_0$ for which the estimates in Lemma \ref{Vk} hold.
From the Ergodic Theorem one can choose a point $\xi_n\in\Omega$ such that
$$\lim_{m\to\infty}\frac{1}{m}\#\{0\leq i<m\colon f^i\xi_n\in V_{k,M}\}=\mu_n(V_{k,M}).$$
%take a generic point $\xi_n$ with respect to $\mu_n$.
%Approximating the indicator function of $V_{k,M}$ by a continuous bump function we have
%$\displaystyle{\lim_{m\to\infty}}(1/m)\#\{0\leq i<m\colon f^i\xi_n\in V_{k,M}\}
%\geq u-\eta$. 
If $\mu_n\neq\delta_Q$, then the positive orbit of $\xi_n$ is a concatenation of segments in $V_{k,M}$ and those out of $V_{k,M}$.
Lemma \ref{Vk} yields
$$\mu_n(\rho_{0,k}\log J^u)=
\lim_{m\to\infty}\frac{1}{m}\sum_{i=0}^{m-1} \rho_{0,k}\log
J^u(f^i(\xi_n))\geq\mu_n(V_{k,M})\cdot \frac{1}{2}\log(4-\varepsilon).$$
Observe that the same inequality remains to hold in the case $\mu_n=\delta_Q$.
Since
$\displaystyle{\varliminf_{n\to\infty}}\mu_n(V_{k,M})\geq u>0$
we get
$$\varliminf_{n\to\infty}\mu_n(\rho_{0,k}\log J^u)\geq \frac{u}{2}\log(4-\varepsilon).$$

If $u\neq1$, then the weak convergence for the sequence $\{\frac{\mu_n-u\delta_Q}{1-u}\}_n\subset\mathcal M(f)$ implies
$$\lim_{n\to\infty}\mu_n(\rho_{1,k}\log J^u)=(1-u)\nu(\rho_{1,k}\log J^u).$$
The same inequality remains true in the case $u=1$.
Consequently,
\begin{align*}\varliminf_{n\to\infty}\lambda^u(\mu_n)\geq
\varliminf_{n\to\infty}\mu_n(\rho_{0,k}\log
J^u)+\lim_{n\to\infty}\mu_n(\rho_{1,k}\log J^u) \geq  \frac{u}{2}\log(4-\varepsilon)+(1-u)\nu(\rho_{1,k}\log J^u).
\end{align*}
Since $\nu\{Q\}=0$, $\rho_{1,k}\log J^u\to \log J^u$ $\nu$-a.e. as $k\to\infty$. Letting $k\to\infty$ and then using the Dominated Convergence Theorem gives
the first estimate in the proposition. A proof of the second one
is completely analogous, with the second inequality in Lemma \ref{Vk}.
\end{proof}

\subsection{Existence of equilibrium measures for $\varphi_t$}\label{proff}
We now complete the proof of the theorem.

%\begin{proof}

%\begin{proof}
%The
%entropy of $\mu$ is written as a linear combination of the entropies
%of its ergodic components. To show the same property for the unstable
%Lyapunov exponents,
%let $\mu\in\mathcal M(f)$, and let $\mu=\int_{\nu\in\mathcal M^e(f)}
%g(\nu)d\tau(\nu)$ denote the ergodic decomposition of $\mu$. We have
%\begin{align*}\lambda^u(\mu)=\lim_{k\to\infty}\mu(\rho_{1,k}\log J^u)
%&=\lim_{k\to\infty}\int_{\nu\in\mathcal
%M^e(f)}g(\nu)\left(\int\rho_{1,k}\log J^ud\nu\right)d\tau(\nu)\\
%&= \int_{\nu\in\mathcal M^e(f)}g(\nu)\lambda^u(\nu)d\tau(\nu).\end{align*} Here, the first equality follows from the Dominated Convergence Theorem,
%and the second one follows from the continuity of $\rho_{1,k}\log J^u$.
%The last one
%from the Dominated Convergence Theorem applied to the sequence of
%functions $\nu\in\mathcal M^e(f)\to g(\nu)\int\rho_{1,k}\log J^ud\nu$.
%It follows that there exists an ergodic
%component $\mu'$ of $\mu$ such that $F_{\varphi_t}(\mu')=F_{\varphi_t}(\mu)$. This implies
%the desired equality.
%\end{proof}

\begin{proof}[Proof of the theorem]
By the ergodic decomposition theorem \cite{Man87},
the unstable Lyapunov exponent of $\mu$ is written as a linear combination of the 
unstable Lyapunov exponents of its ergodic components. 
Since the same property holds for entropies and $\mathcal M(f)$ is compact,
one can choose 
a convergent sequence $\{\mu_n\}\subset\mathcal M^e(f)$
such that $F_{\varphi_t}(\mu_n)>P(t)-1/n$.
Let $\mu\in\mathcal M(f)$ denote the limit point.
In the case $t\leq0$, the upper semi-continuity of entropy and Proposition \ref{lyap2} yield $P(t)=\displaystyle{\lim_{n\to\infty}}F_{\varphi_t}(\mu_n)\le F_{\varphi_t}(\mu)$. Namely $\mu$ is an equilibrium measure for $\varphi_t$.

We now consider the case $t>0$.
Write $\mu=u\delta_Q+(1-u)\nu$ where $0\leq u\leq 1$, $\nu\in\mathcal M(f)$ and $\nu\{Q\}
=0$. The upper semi-continuity
of entropy gives
$$P(t)
=\lim_{n\to\infty}F_{\varphi_t}(\mu_n)\leq
h(\mu)-t\varliminf_{n\to\infty}\lambda^u(\mu_n).$$
If $u=1$ then $\mu=\delta_Q$ and thus $h(\mu)=0$.
Proposition \ref{lyap2} gives
$P(t)
\leq -(t/2)\log (4-\varepsilon)$
and a contradiction arises because 
$P(t)>-(t/2)\log(4-\varepsilon)$ from \eqref{t0} and $t<t_0$.
Hence $u\neq1$ holds. If $u\neq0$, then
using Proposition \ref{lyap2} and $h(\mu)=(1-u)h(\nu)$ we have
\begin{align*}
P(t)&\leq h(\mu)-t\left(\frac{u}{2}\log(4-\varepsilon)+(1-u)\lambda^u(\nu) \right)\\
&=(1-u)F_{\varphi_t}(\nu)- \frac{tu}{2}\log(4-\varepsilon) <
(1-u)F_{\varphi_t}(\nu)+uP(t).
\end{align*}
Rearranging this gives $(1-u)P(t)< (1-u) F_{\varphi_t}(\nu)$, and thus
$P(t)<F_{\varphi_t}(\nu)$, a contradiction. Hence $u=0$, and
$P(t)\leq F_{\varphi_t}(\nu)= F_{\varphi_t}(\mu)$. Namely $\nu$ is an equilibrium measure for $\varphi_t$.
\end{proof}

\section*{Appendix: on the size of $t_0$.}
 Since the topological entropy of $f$ is $\log2$, the Variational Principle shows
$P(0)=\log2$. By Ruelle's inequality \cite{Rue78}, $P(1)\leq0$. Since $f$ has no SRB measure \cite{Tak12}, $P(1)<0$.
Hence, there equation $P(t)=0$ has the unique solution in $(0,1)$, which is denoted by $t^u$.
Observe that $t^u<t_0$.
From the next lemma and the fact that $t^u\to1$ as $b\to0$ \cite[Theorem B]{SenTak12} it follows that
$t_0$ can be made arbitrarily large by choosing sufficiently small $\varepsilon$ and $b$.

\begin{lemma}\label{t0size}
$t_0\geq\frac{\log2}{(1/t^u)\log2-(1/2)\log(4-\varepsilon)}$.

\end{lemma}
\begin{proof}
Consider the pressure function $t\in\mathbb R\mapsto P(t)$ and its graph.
The two points $(0,\log2)$ and $(t^u,0)$ lie on the graph.
Since the graph is concave up, $\{(t,P(t))\colon t>t^u\}$ lies the above
of the straight line through the two points. In other words, $P(t)>-(1/t^u)(t-t^u)\log2+\log2$.
A direct computation shows that $-(1/t^u)(t-t^u)\log2+\log2>-(t/2)\log(4-\varepsilon)$ provided
$t<\frac{\log2}{(1/t^u)\log2-(1/2)\log(4-\varepsilon)}$.
\end{proof}

\subsection*{Acknowledgments}
We thank anonymous referees for useful comments. 
S. S.  is partially supported by the CNPq Brazil. 
H. T. is partially supported by the Grant-in-Aid for Young Scientists (B) of the JSPS, Grant No.23740121. This research is partially supported by the Kyoto University Global COE Program.
We thank Renaud Leplaideur, Isabel Rios, Paulo Varandas, and Michiko Yuri for fruitful discussions.

\end{document}